\documentclass[12pt]{amsart}

  \usepackage{amssymb}
  \usepackage{a4}
  \usepackage{graphicx}
\usepackage{caption}
\usepackage{subcaption}

  \theoremstyle{plain}
    \newtheorem{theorem}{Theorem}[section]
    \newtheorem{proposition}[theorem]{Proposition}
    \newtheorem{lemma}[theorem]{Lemma}
    \newtheorem{corollary}[theorem]{Corollary}
  \theoremstyle{definition}

    \newtheorem{remark}[theorem]{Remark}

  \newcommand{\depth}{\mathrm{depth}}
  \newcommand{\reg}{\mathrm{reg}}
   \newcommand{\pd}{\mathrm{pd}}
   \newcommand{\height}{\mathrm{height}}
  
   \newcommand{\lk}{\mathrm{lk}}
 \newcommand{\F}{\mathcal{F}}
\newcommand{\M}{\mathcal{M}} 
\newcommand{\LD}{\mathrm{LD}} 
  
  \title{Algebraic properties of classes of path ideals}
  \author{Martina Kubitzke \and Anda Olteanu}
  
 \thanks{The second author was supported by the strategic grant POSDRU/89/1.5/S/58852, Project ``Postdoctoral program for training scientific researchers" co-financed by the European Social Fund within the Sectorial Operational Program Human Resources Development 2007 - 2013"}

\address{Institut f\"ur Mathematik Goethe-Universit\"at
Robert-Mayer-Str. 10, Frankfurt am Main, Frankfurt, Germany} \email{kubitzke@math.uni-frankfurt.de} 
\address{Faculty of Mathematics and Computer Science, Ovidius University, Bd.\ Mamaia 124,
 900527, Constanta, Romania,} 
\address{University ``Politehnica" of Bucharest, Faculty of Applied Sciences,
Splaiul Independen\c tei, No.
313, 060042, Bucharest, Romania ,}\email{olteanuandageorgiana@gmail.com}

\begin{document}

\maketitle

\begin{abstract}  We consider path ideals associated to special classes of posets such as tree posets and cycles. We express their property of being sequentially Cohen-Macaulay in terms of the underlying poset. Moreover, monomial ideals, which arise from the Luce-decomposable model in algebraic statistics, can be viewed as path ideals of certain posets. We study invariants of these so-called \emph{Luce-decomposable} monomial ideals for diamond posets and products of chains. In particular, for these classes of posets, we explicitly compute their Krull dimension, their projective dimension, their regularity and their Betti numbers.\\



\end{abstract}

  \maketitle
  
\section{Introduction}\label{sect:Setting}
In this paper, we study path ideals for the Hasse diagram of posets. If the Hasse diagram is a tree as an undirected graph, these ideals can be seen as generalization of path ideals for 
rooted trees, as studied in \cite{BHK,HT}. Let $G=(V,E)$ be a directed graph 
on vertex set $V=\{x_1,\ldots,x_n\}$ and edge set $E$ and let $k$ be an arbitrary field. Consider the polynomial ring $S=k[x_1,\ldots,x_n]$,
where we identify the vertices of $G$ and the variables of $S$. 
The \emph{path ideal of $G$ of length $t$} is the monomial ideal
\begin{equation*}
 I_t(G)=\langle x_{i_1}\cdots x_{i_t}~:~x_{i_1},\ldots,x_{i_t} \mbox{ is a path of length }t\mbox{ in } G\rangle\subseteq k[x_1,\ldots,x_n].
\end{equation*}
Path ideals have been introduced by Conca and De Negri \cite{CD} and they generalize arbitrary edge ideals of graphs \cite{Villa}. 
Since then, path ideals have attracted the attention of a lot of researchers and they are fairly well-studied for special classes of graphs such as the line 
graph and the cycle \cite{AF,SKT} and also for rooted trees \cite{BHK,CD,HT}. In the study of path ideals of rooted trees an orientation of the edges was implicitly assumed (by directing 
the edges from the root vertex to the leaves). In this paper, we want to pursue this direction of research further. Given a partially ordered set $P$ (poset 
for short) we interpret its Hasse diagram as a directed graph, where edges are oriented from bottom to top. One can associate 
the path ideals to $P$, where paths are directed, i.\,e., they correspond to saturated chains of the poset of a given length.
The goal of this article is to study algebraic properties of this class of ideals and to relate them to combinatorial properties of the underlying poset. 
Our first result provides a lower bound for the regularity of these ideals. This bound does not require any assumptions on the poset and moreover, it 
turns out to be sharp (Proposition \ref{prop:regularity}). In Section \ref{sect:seqCM}, we study the path ideals of posets, whose Hasse diagrams are   trees (as undirected graphs) or cycles. We emphasize that we do not require trees to be rooted anymore. We show that path ideals of these not  necessarily rooted trees, so-called ``tree posets'', are sequentially Cohen-Macaulay (Corollary \ref{cor:seqCM}), which is a generalization of the corresponding result for rooted treed 
\cite[Corollary 2.12]{HT}. As in \cite{HT}, this will follow from the stronger result that the facet complex of the path ideal of a tree poset is 
a simplicial tree (Theorem \ref{thm:simpForest}) and \cite[Corollary 5.6]{Fa1}. If the Hasse diagram of $P$ is a cycle, we classify when the 
facet complexes of the corresponding path ideals are simplicial trees (Theorem \ref{thm:cycle}), which allows us to characterize those path ideals of 
a cycle, which are sequentially Cohen-Macaulay (Corollary \ref{cor:cycleCM}). Moreover, for path ideals of cycles, we compute their height (Proposition \ref{prop:height}). In Section \ref{sect:specialClasses}, we slightly change our perspective and consider the defining ideals 
of the Luce-decomposable model --~a frequently studied toric model in algebraic statistics~-- for special types of posets \cite{Cs2,Cs1,SW}. Those ideals can be interpreted as the path ideals generated 
by the maximal chains of the underlying poset, with the difference that edges are labelled instead of vertices, i.e., variables correspond to edges 
rather than vertices. More precisely, in Section 
\ref{subsect:diamond}, we study posets, that, due to their appearence, will be referred to as \emph{diamond posets} and in Section 
\ref{subsect:chains} we investigate products of two chains of length $2$ and $n$, respectively. 
Our main results of Section \ref{sect:specialClasses} are formulas for the depth, Krull dimension, projective dimension and the $\mathbb{Z}$-graded 
Betti numbers for the defining ideal of the Luce-decomposable model when the underlying poset is a diamond poset (Theorem \ref{thm:diamond}, 
Corollary \ref{cor:diamond}, Lemma \ref{lem:diamond}) or the product of 
two chains (Theorem \ref{thm:chains}, Corollary \ref{cor:chains}, Lemma \ref{lem:chains}).


\section{Background and notation}
In this section, we provide some background and fix some notation, which will be used throughout this article. 
\subsection{Path ideals of posets}
We will denote by $(P,<_P)$ a partially ordered set (\emph{poset}, for short) on the ground set $\{x_1,\ldots,x_n\}$. Given two elements $x,y\in P$, 
we say that $y$ \emph{covers} $x$ if $x<_P y$ and there is no $z\in P$ with $z\neq x,y$ such that $x<_P z<_P y$. A sequence $x_{i_1},\ldots,x_{i_r}$ of 
elements in $P$ is called a \emph{saturated chain} if $x_{i_j}$ covers $x_{i_{j-1}}$ for $2\leq j\leq r$. Throughout this paper, we always 
assume a chain to be saturated, even though we do not always state this explicitly. A chain of $P$ is called \emph{maximal} if it starts and ends 
with a minimal and maximal element, respectively. To a poset $P$ we associate 
its \emph{Hasse diagram}, which is the directed graph on vertex set $\{x_1,\ldots, x_n\}$ such that $\{x_i,x_j\}$ is an edge if and only if 
$x_j$ covers $x_i$. In the following, we will use $P$ to denote both, the poset and its Hasse diagram and it will be clear from the 
context to which object we refer to. Note that in the literature the Hasse diagram of a poset is mostly interpreted as an undirected graph. However,
it is natural to choose an orientation from bottom to top. 
A \emph{path} of length $t\geq2$ in (the Hasse diagram of) $P$ is a sequence of $t$ vertices $x_{i_1},\ldots,x_{i_t}$ such that $\{x_{i_j},x_{i_{j+1}}\}$ is 
an edge for $1\leq j\leq t-1$. In the language of posets, a path of length $t$ is a saturated chain of length $t$. 
Let $k$ be an arbitrary field and let $S=k[x_1,\ldots,x_n]$ the polynomial ring in $n$ variables, where we identify the variables with the elements of 
a given poset $P$.
The \emph{path ideal} of $P$ of length $t\geq2$ is the squarefree monomial ideal
\begin{equation*}
 I_t(P)=\langle x_{i_1}\cdots x_{i_t}~:~x_{i_1},\ldots,x_{i_t} \mbox{ is a path of length }t\mbox{ in } P\rangle\subseteq k[x_1,\ldots,x_n].
\end{equation*}
One may note that, for $t=2$, $I_t(P)$ is the usual edge ideal.

Given two posets $P$ and $Q$ on ground sets $\{x_1,\ldots,x_n\}$ and $\{y_1,\ldots,y_s\}$, respectively, we define the connected sum $P\oplus Q$ of $P$ and $Q$ as the poset on the ground set $\{x_1,\ldots,x_n,y_1,\ldots,y_s\}$, whose ordering relation is given by $x<_{P\oplus Q}y$ if either $x<_P y$, or $x<_Q y$ or if $x\in P$ and $y\in Q$.

\subsection{Simplicial complexes}
In this section, we will recall two canonical ways of how to associate a simplicial complex to a squarefree monomial ideal and vice versa. 
Here and in the sequel, we set $S=k[x_1,\ldots,x_n]$. 
An \emph{abstract simplicial complex} $\Delta$ on $[n]=\{1,\ldots,n\}$ is a collection
of subsets of $[n]$ such that $G \in \Delta$ and $F\subseteq G$ imply $F\in\Delta$. 
The elements of $\Delta$ are called \emph{faces}. Inclusionwise maximal and $1$-element faces are called \emph{facets}
and \emph{vertices}, respectively. We denote the set of facets of $\Delta$ by $\F(\Delta)$. Subsets $H$ of $[n]$ such that $H\notin \Delta$ are called \emph{non-faces}.\\
For $W\subseteq [n]$, we denote by 
\begin{equation*}
\Delta_W=\{F\in \Delta~:~F\subseteq W\}
\end{equation*}
the \emph{restriction} of $\Delta$ to $W$. For a face $F\in \Delta$ the simplicial complex
\begin{equation*}
 \lk_{\Delta}(F)=\{G\subseteq [n]~:~ F\cup G\in \Delta,\; F\cap G=\emptyset\}
\end{equation*}
is called the \emph{link} of $F$ in $\Delta$. 
If $\Delta$, $\Gamma$ are simplicial complexes on disjoint sets of vertices, then the simplicial complex
\begin{equation*}
\Delta \ast \Gamma=\{F\cup G~:~F\in \Delta,\; G\in \Gamma\}
\end{equation*}
is called the \emph{(simplicial) join} of $\Delta$ and $\Gamma$. \\
Given a simplicial complex $\Delta$ on vertex set $[n]$, we can study its \emph{Stanley-Reisner ideal} $I_{\Delta}$
\begin{equation*}
I_\Delta=\langle \prod_{i\in H}x_i~:~H\notin \Delta\rangle\subseteq S,
\end{equation*}
which is determined by the non-faces of $\Delta$. 
The quotient $S/I_{\Delta}$ is called the \emph{Stanley-Reisner ring} of $\Delta$. 
We obtain a $1-1$-correspondence between simplicial complexes on $[n]$ and squarefree monomial ideals in $S$ by assigning to a squarefree 
monomial ideal $I\subseteq S$ the simplicial complex $\Delta_I$, whose minimal non-faces correspond to the minimal generators of $I$, i.\,e., $I_\Delta=I$. 
Following \cite{AF} we call $\Delta_I$ the \emph{Stanley-Reisner complex} of $I$. For more details on simplicial complexes and Stanley-Reisner rings, one may see \cite{BH} for instance.\\
Another way to construct a squarefree monomial ideal from a given simplicial complex $\Delta$ is provided by the \emph{facet ideal} $I(\Delta)$, defined by
\begin{equation*}
I(\Delta)=\langle \prod_{i\in F}x_i~:~F\in\F(\Delta)\rangle.
\end{equation*}
Vice versa, given a squarefree monomial ideal $I$, the simplicial complex $\Delta(I)$ whose facets correspond to the minimal generators of $I$ is uniquely determined and it is referred to as the \emph{facet complex} of $I$. 
If, as in our context, $I$ is the path ideal of a poset $P$, then we only write $\Delta_t(P)$ for $\Delta(I_t(P))$ and similarly, $\Delta_{t,P}$ for 
the Stanley-Reisner complex $\Delta_{I_t(P)}$.

Since the Hochster formula for the Betti numbers of the Stanley-Reisner ring of a simplicial complex will be crucial for the computation of the Betti numbers in Section \ref{sect:specialClasses}, we now recall it in the form which serves our purposes best. 

\begin{theorem}\cite[Theorem 5.5.1]{BH}\label{Hochster}
Let $\Delta$ be a simplicial complex on vertex set $[n]$. 
Then, for $0\leq i\leq \pd (k[\Delta])$ and $j\in \mathbb{N}$ it holds that
\begin{equation*}
\beta_{i,j}(k[\Delta])=\sum_{\substack{W\subseteq [n]\\ |W|=j}}\dim_k\widetilde{H}_{j-i-1}(\Delta_W;k).
\end{equation*}
\end{theorem}

For the computation of the homology of certain subcomplexes as they appear in the Hochster formula, we will repeatedly apply the following Mayer-Vietoris sequence.

\begin{theorem} (see e.g., \cite[Theorem 25.1]{Munkres})
Let $\Delta$ be a simplicial complex and let $\Delta_1$ and $\Delta_2$ be subcomplexes of $\Delta$ such that $\Delta=\Delta_1\cup\Delta_2$. Then there is a long exact sequence 
\begin{equation}\label{eq:MV}
\cdots \rightarrow H_p(\Delta_1\cap \Delta_2)\rightarrow H_p(\Delta_1)\oplus H_p(\Delta_2)\rightarrow H_p(\Delta)\rightarrow H_{p-1}(\Delta_1\cap \Delta_2)\rightarrow \cdots,
\end{equation}
where the homology can be taken over any field. The above sequence is called the \emph{Mayer-Vietoris sequence} of $(\Delta_1,\Delta_2)$.
\end{theorem}

Another useful fact we will need is the K\"unneth formula for the homology of a join of two simplicial complexes.

\begin{theorem}(see e.g.,\cite{Hatcher}) \label{thm:join}
Let $k$ be an arbitrary field. Let $\Delta$ be a simplicial complex such that all homology groups of $\Delta$ are free $k$-modules. Let $\Gamma$ be any simplicial complex. Then
\begin{equation*}
\widetilde{H}_p(\Delta\ast \Gamma;k)\cong \bigoplus_{i+j=p-1}\widetilde{H}_i(\Delta;k)\otimes_k\widetilde{H}_j(\Gamma;k).
\end{equation*}
\end{theorem}


\section{Regularity of path ideals}\label{sect:regularity}
In this section we study the regularity of path ideals of a poset $P$. 
Since one cannot expect to find an explicit formula for the general case, we focus on finding a lower bound. 
In doing so, we complement the work of Bouchat, H{\`a} and O'Keefe \cite{BHK} who derive an upper bound for the regularity of path ideals of rooted trees. As will be pointed out, 
there exist cases where their bound and our bound coincide and thus are sharp.

For a poset $P$ and an integer $t\geq 2$, 
let $a_t(P)$ be the maximal integer $s$ such that there exists a set $\{P_1,\ldots,P_s\}$ of pairwise disjoint paths of length $t$ in $P$, 
such that $\bigcup_{i=1}^s P_i$ does not contain any other paths of length $t$ in $P$. 

\begin{proposition}\label{prop:regularity}
Let $P$ be a poset on $\{x_1,\ldots,x_n\}$, $S=k[x_1,\ldots,x_n]$ and $t\geq 2$ be an integer. Then
\begin{equation*}
\reg(S/I_t(P))\geq a_t(P)\cdot (t-1).
\end{equation*}
\end{proposition}

\begin{proof}
Let $a_t(P)=s$. By assumption, there exists a set $\{P_1,\ldots,P_s\}$ of pairwise disjoint paths of length $t$ in $P$ 
such that $\bigcup_{i=1}^s P_i$ does not contain any other paths of length $t$ in $P$. 
We claim that $\widetilde{H}_{s(t-1)-1}(\Delta_t(P)_{\bigcup_{i=1}^s P_i};k)=k$. Since $I_{\Delta_{t,P}}=I_t(P)$, we can then infer from Theorem  
\ref{Hochster} that $\beta_{s,st}(S/I_t(P))\neq 0$, which implies the desired inequality.\\
First note that $(\Delta_{t,P})_{\bigcup_{i=1}^s P_i}=\ast_{i=1}^s (\Delta_{t,P})_{P_i}$. 
Indeed, for $1\leq i\leq s$, let $F_i$ be a face of $(\Delta_{t,P})_{P_i}$. Since each $P_i$ is a minimal non-face of $\Delta_{t,P}$, 
we have that $P_i\not\subseteq F_i$ for all $1\leq i\leq s$. Moreover, as $P_1,\ldots,P_s$ are pairwise disjoint, it even holds 
that $P_i\not\subseteq \bigcup_{j=1}^s F_j$ for all $1\leq i\leq s$. This, together with the facts that $\bigcup_{j=1}^s F_j\subseteq \bigcup_{j=1}^s P_j$ 
and $P_1,\ldots,P_s$ are the only paths of length $t$ in $P$ contained in the latter set, implies that $\bigcup_{j=1}^s F_j$ does not contain 
any path of length $t$ of $P$. Since the latter ones are non-faces of $\Delta_{t,P}$, we infer 
that $\bigcup_{j=1}^s F_j\in(\Delta_{t,P})_{\bigcup_{i=1}^s P_i}$. 
This shows $\ast_{i=1}^s(\Delta_{t,P})_{P_i}\subseteq (\Delta_{t,P})_{\bigcup_{i=1}^s P_i}$. 
The other inclusion follows from similar arguments. 
So as to compute the homology of $\ast_{i=1}^s(\Delta_{t,P})_{P_i}$, note that for $1\leq i\leq s$ the complex $(\Delta_{t,P})_{P_i}$ is the boundary of a $(t-1)$-simplex and as such homeomorphic to a $(t-2)$-sphere. Theorem \ref{thm:join} implies
\begin{align*}
\widetilde{H}_i(\Delta_t(P)_{\bigcup_{i=1}^s P_i};k)=
\begin{cases}
&k, \mbox{ if } i=s(t-1)-1\\
&0, \mbox{ otherwise}.
\end{cases}
\end{align*}
\end{proof}

We will now compare the just derived lower bound with the upper bound from \cite{BHK}.

\begin{remark}
\begin{itemize}
\item[(i)] Let $P=L_n$ be a line graph on $n$ vertices and let $n=s(t+1)+r$ for a certain $0\leq r\leq t$ and $s\in \mathbb{N}$. From \cite[Theorem 3.4]{BHK} we 
obtain $\reg(S/I_t(L_n))\leq (t-1)(1+s)$ whereas our bound yields 
\begin{equation*}
\reg(S/I_t(L_n))\geq
\begin{cases}
 (t-1)s, \qquad\qquad\mbox{if }r<t\\
(t-1)(s+1), \qquad \mbox{if } r=t.\\
\end{cases}
\end{equation*}
Hence, for $n=s(t+1)+t$, both bound coincide and are indeed sharp, i.e., $\reg(S/I_t(L_n))=(t-1)(s+1)$ for $n=s(t+1)+t$.
\item[(ii)] A special case of (i) is given when $P$ is a chain of length $t$. In this case, $\reg(S/I_t(P))=t-1$, which also follows from an easy
computation. Moreover, if $P$ is a disjoint union of $s$ chains of length $t$, 
then the lower bound given in the above proposition coincides with the upper bound 
from \cite[Theorem 3.4]{BHK}. Hence, --~as also already remarked in \cite[Remark 3.5]{BHK}~-- both bounds are sharp in this case. 
\item[(iii)] Combining the discussion in (i) and (ii), we obtain that the bounds from Proposition \ref{prop:regularity} and \cite[Theorem 3.4]{BHK}
coincide if $P$ is a disjoint union of chains of length $s(t+1)+t$ and hence yield the right regularity in these cases.
\end{itemize}
\end{remark}

\section{Path ideals of tree posets, lines, and cycles}
In this section, we study path ideals of three different classes of posets, i.e., tree posets, lines and cycle posets. 
One of our main results shows that path ideals of tree posets are sequentially Cohen-Macaulay. This generalizes Theorem 1.1 in \cite{HT} 
for directed trees. Moreover we can characterize those lines that are Cohen-Macaulay, thereby supplementing with results in \cite{AF} and we can determine classes of cycle posets whose path ideal is sequentially Cohen-Macaulay.

\subsection{Tree posets: A class of sequentially Cohen-Macaulay posets}\label{sect:seqCM}

Path ideals of rooted trees have been studied by Bouchat, H{\`a} and O'Keefe \cite{BHK}, but also by He, 
VanTuyl \cite{HT} and by Alilooee, Faridi \cite{AF} for line graphs. He and Van Tuyl \cite{HT} in particular 
showed that path ideals of rooted trees are sequentially Cohen-Macaulay. 
 They noticed that each (undirected) tree can be considered as a rooted tree by choosing any vertex as the root vertex and 
by orienting the edges ``away'' from this root. Generalizing the work of He and Van Tuyl, we investigate path ideals of posets whose Hasse
diagram --~viewed as an undirected graph~-- is a tree. By doing this, we loose the freedom of picking a root vertex and orienting the edges. 
Indeed, the direction of the edges is naturally induced by the order relations of the poset and, in general, the directed tree will not be rooted.  
We show that even in this case, the corresponding path ideals are sequentially Cohen-Macaulay.

The definition of a tree poset requires the definition of the following forbidden substructure. 
We say that a poset $P$ is a \textit{pencil} if its elements can be written as a disjoint 
union $V=\{x_1,\ldots,x_r\}\cup\{z\}\cup\{y_1,\ldots,y_m\}$, for some $r,m\geq2$, such that 
$x_i<_P z<_P y_j$, for all $1\leq i\leq r$, $1\leq j\leq m$, and the sets $\{x_1,\ldots,x_s\}$ and $\{y_1,\ldots,y_m\}$ are both antichains.

A poset $P$ will be called a \textit{tree (poset)} if its Hasse diagram is a tree (as an undirected graph) and if it does not contain any pencil as a subposet. 
If a poset $P$ is a disjoint union of tree posets, then $P$ will be called \textit{a forest (poset)}.

Borrowing terminology from graph theory, we say that an element $x$ of a tree poset $ P$ is a \textit{leaf} if --~as a node in the undirected graph $P$~--
it has degree $1$. Observe that this means that a leaf is either a maximal element of $P$ or it is a minimal element, that is covered by exactly one other
element of $P$. 
  
The main concept we will make use of in the following is that of a simplicial tree, originally introduced by Faridi \cite{Fa}. 
Given a simplicial complex $\Delta$, a facet $F\in\Delta$ is called a \textit{leaf} if either $F$ is the only facet of $\Delta$ 
or there exists a facet $G\in\Delta$ such that $F\cap H\subseteq F\cap G$ for every $F\neq H\in\mathcal{F}(\Delta)$. 
A simplicial complex $\Delta$ is a \textit{simplicial tree} if it is connected and if every non-empty subcomplex, generated by a subset of $\mathcal{F}(\Delta)$ 
 contains a leaf. By convention, the empty simplicial complex is a simplicial tree. Moreover, a simplicial complex, all of whose connected components are simplicial 
trees, is called a \textit{simplicial forest}.

Faridi \cite[Corollary 5.6]{Fa} and subsequently Van Tuyl and Villareal \cite[Corollary 5.7]{VTV} established the following connection between simplicial trees and sequentially Cohen-Macaulay modules:

\begin{theorem}\cite[Corollary 5.7]{VTV}\label{simpforest} 
Let $\Delta$ be a simplicial complex on the vertex set $[n]$ and $S=k[x_1,\ldots,x_n]$ be a polynomial ring over a field $k$. 
If $\Delta$ is a simplicial forest and $I=I(\Delta)$, then $S/I$ is sequentially Cohen-Macaulay over $k$. 
\end{theorem}

 He and Van Tuyl proved that, if a graph $G$ is a rooted tree, then $\Delta_t(G)$ is a simplicial tree and, by the above result, in particular sequentially 
Cohen-Macaulay. We will generalize their result to forest posets. Firstly, we provide some preparatory results.

\begin{lemma}\label{leaf} Let $P$ be a tree poset. Then there exists a maximal chain $L$ in $P$ such that its maximal element or its minimal 
element does not belong to any other maximal chain of $P$. 
\end{lemma}

\begin{proof} Let $\{x_{\alpha_1},\ldots,x_{\alpha_r}\}$ and $\{x_{\beta_1},\ldots,x_{\beta_s}\}$ be the set of minimal and maximal elements of $P$, respectively.
 Obviously, every maximal chain of $P$ contains a unique element from each of those two sets. 
Let $L_1$ be a maximal chain of $P$ and assume that $x_{\alpha_{i_1}}\in L_1$ and $x_{\beta_{j_1}}\in L_1$. 
If $L_1$ is the only maximal chain that contains $x_{\beta_{j_1}}$, then we finished. 
Otherwise, there exists a maximal chain $L_2$ of $P$ such that $x_{\beta_{j_1}}\in L_2$ and $x_{\alpha_{i_2}}\in L_2$, for some $i_2\in\{1,\ldots,r\}$. 
One may note that $i_2\neq i_1$ since the Hasse diagram of $P$ is a tree as an undirected graph. If --~apart from $L_2$~-- there is no other 
maximal chain in $P$ containing $x_{\alpha_{i_2}}$, then we are done. If this is not the case, there must exist a maximal chain $L_3$ in $P$ such 
that $x_{\alpha_{i_2}}\in L_3$ and $x_{\beta_{j_2}}\in L_3$, for some $j_2\in\{1,\ldots,s\}$. One observes that $x_{\beta_{j_1}}\notin L_3$, 
since the Hasse diagram of $P$ is a tree as an undirected graph. As before, if $L_3$ is not the only maximal chain of $P$ that 
contains $x_{\beta_{j_2}}$, then there exists a maximal chain $L_4$ such that $x_{\beta_{j_2}}\in L_4$ and $x_{\alpha_{i_3}}\in L_4$, for some 
$i_3\in \{1,\ldots,r\}$. Obviously, $i_3\neq i_2$ since the Hasse diagram of $P$ is a tree as an undirected graph. 
Moreover, $i_3\neq i_1$ since otherwise, $P$ would contain a cycle or a pencil as a subposet. 
Indeed, assume that $i_3=i_1$ and that $P$ does not contain any cycle. Then there must be an element $\omega\in P$ such that 
$x_{\alpha_{i_1}}<_P\omega<_Px_{\beta_{j_1}}$ and $x_{\alpha_{i_2}}<_P\omega<_P x_{\beta_{j_2}}$, i.e., $P$ contains a pencil, 
which is a contradiction. Being $P$ a finite poset, the described process must finish after a finite number of steps. 
Therefore, we will eventually find a maximal or a minimal element that belongs to a unique maximal chain.
\end{proof}

\begin{remark}\rm\label{maxmin} 
One may note that we may always assume that a tree poset $P$ has maximal paths of the form $x_{i_1},\ldots,x_{i_t}$ such that $x_{i_t}$ is a leaf. 
Indeed, if for any maximal path of length $t$, the maximal element is not a leaf, then by Lemma~\ref{leaf}, 
there is a maximal chain $C=x_{i_1},\ldots,x_{i_t}$ such that the minimal element $x_{i_1}$ belongs to no other maximal chain than $C$. 
Moreover, since, as an undirected graph, $P$ is a forest and hence contains no cycle, $x_{i_1}$ needs to be a leaf. 
In this case, we may consider the dual poset $\bar P$, that is the poset on the same set of elements with the order relations reversed, i.e.,  
$x<_{\bar{P}}y$ if and only if $y<_P x$. Clearly, $I_t(P)=I_t(\bar{P})$ and $\Delta_t(P)=\Delta_t(\bar {P})$. 
In addition, $x_{i_t},\ldots,x_{i_1}$ is a maximal path in $\bar{P}$ such that $x_{i_1}$ is a leaf.   
 \end{remark}

The above remark allows us to assume that each tree poset $P$ contains a maximal element that belongs to a unique maximal chain.

 \begin{lemma}\label{intersection} Let $P$ be a tree poset, $t\geq2$ and let $F=x_{i_1},\ldots,x_{i_t}$ and $G=x_{j_1},\ldots,x_{j_t}$ be two paths in $P$. 
If $F\cap G\neq\emptyset$, then $F\cap G$ is a connected path $x_{k_1},\ldots,x_{k_r}$, with $r\leq t$ and $x_{i_1}\leq_P x_{k_1}$, $x_{j_1}\leq_{P}x_{k_1}$.
 \end{lemma}
 \begin{proof} Let $F\cap G=x_{k_1},\ldots,x_{k_r}$. If $P$ is a rooted tree, then the statement follows from \cite[Lemma 2.5]{HT}. 
Suppose now that $P$ is not a rooted tree. If there exists $w\in P$ such that $x_{i_1}\geq_P w$ and $x_{j_1}\geq_P w$, then again the statement 
follows from \cite[Lemma 2.5]{HT}. Therefore we may assume that there is no $w\in P$ such that $x_{i_1}\geq_P w$ and $x_{j_1}\geq_P w$. 
Since $F\cap G\subset F$ and $F\cap G\subset G$, it is clear that $x_{i_1}\leq x_{k_1}$ and $x_{j_1}\leq x_{k_1}$. We have to prove 
that $F\cap G$ is a connected path. Assume by contradiction that there exist two distinct paths $x_{\alpha_1},\ldots,x_{\alpha_s}$ and 
$x_{\alpha_{s+1}},\ldots,x_{\alpha_m}$ in $F\cap G$ such that there is no edge between $x_{\alpha_s}$ and $x_{\alpha_{s+1}}$ (i.e., $x_{\alpha_{s+1}}$ does not cover $x_{\alpha_s}$). 
Since 
$\{x_{\alpha_1},\ldots,x_{\alpha_s}\}\cup\{x_{\alpha_{s+1}},\ldots,x_{\alpha_m}\}\subseteq F$ and $F$ is a connected path, 
there are $x_{i_{\beta_1}},\ldots,x_{i_{\beta_{s'}}}$ in $F\setminus G$ such that 
$x_{\alpha_s}\leq_P x_{i_{\beta_1}}\leq_P\cdots\leq_P x_{i_{\beta_{s'}}}\leq_Px_{\alpha_{s+1}}$. By the same reasoning  
there exist $x_{j_{\gamma_1}},\ldots,x_{j_{\gamma_{s''}}}$ in $G\setminus F$ such that 
$x_{\alpha_s}\leq_P x_{j_{\gamma_1}}\leq_P\cdots\leq_P x_{j_{\gamma_{s''}}}\leq_P x_{\alpha_{s+1}}$. 
Therefore, we found two distinct paths between $x_{\alpha_s}$ and $x_{\alpha_{s+1}}$, which is impossible since $P$ is a tree.
 \end{proof}
 
 \begin{lemma}\label{lem:leaf} Let $P$ be a tree poset, $t\geq2$ and $x_{i_1},\ldots,x_{i_t}$ be a path in $P$ such that $x_{i_t}$ is a maximal element that is contained in a unique maximal chain. Then the facet $\{x_{i_1},\ldots,x_{i_t}\}$ is a leaf in the simplicial complex $\Delta_t(P)$.
 \end{lemma}
 \begin{proof} If $P$ is a rooted tree, then the statement follows from \cite[Lemma 2.6]{HT}. Let us assume now that $P$ is not a rooted tree.
 
 Firstly, suppose that $x_{i_1}$ is a minimal element of $P$. In this case, $x_{i_1},\ldots,x_{i_t}$ is a maximal chain. We set $F=\{x_{i_1},\ldots,x_{i_t}\}\in\Delta_t(P)$. If $F\cap G=\emptyset$ for all $G\in\mathcal{F}(\Delta_t(P))\setminus \{F\}$, then $F$ is a leaf. Therefore, we may assume that there is some facet $G\in\Delta_t(P)$ such that $G\cap F\neq \emptyset$. This in particular implies that the set 
 \[T=\bigcup_{\substack{H\in\mathcal{F}(\Delta_t(P)),\\ H\neq F}}H\cap F
 \]
is not empty. By our assumption, we must have, $x_{i_t}\notin T$. Moreover, one may note that, since $x_{i_t}$ is contained in a unique maximal chain, $x_{i_1}$ has to be the unique minimal element of $P$ such that $x_{i_1}<_P x_{i_t}$. Otherwise $x_{i_t}$ would be contained in more than one maximal chain. Hence, for all facets $H\in \Delta_t(G)$ with $F\cap H\neq \emptyset$ we must have $x_{i_1}\in H$. Moreover, by Lemma \ref{intersection}, $F\cap H$ is of the form $\{x_{i_1},x_{i_2},\ldots,x_{i_s}\}$ for some $1\leq s<t$ and all $H\in \mathcal{F}(\Delta)$ with $F\cap H\neq \emptyset$. Obviously, the facet for which $s$ is maximal, is a leaf of $\Delta_t(G)$.
 
 If $x_{i_1}$ is not a minimal element, then there is an element $x_j\in P$, $x_j\neq x_{i_1}$ such that $x_j<_P x_{i_1}$. In this case, let $H=\{x_j,x_{i_1},\ldots,x_{i_{t-1}}\}\in\Delta_t(P)$. Since $x_{i_t}$ belongs only to $F$, we must have $F\cap G\subseteq F\cap H$ for all facets $G\in\Delta_t(P)$, $G\neq F$, which implies that $F$ is a leaf of $\Delta_t(P)$.
 \end{proof}
 We may now prove the following result: 

 \begin{theorem} \label{thm:simpForest}
Let $P$ be a forest poset and let $t\geq2$. Then $\Delta_t(P)$ is a simplicial forest.
 \end{theorem}
 \begin{proof} We use induction on the number of facets of $\Delta_t(P)$. If $\Delta_t(P)$ has only one facet (which occurs if $P$ is just a chain of $t$ elements), then obviously it is a simplicial tree. We assume that the statement holds for all forest posets $Q$ such that $\Delta_t(Q)$ has at most $r-1$ facets.
 
  Let $P$ be a forest poset such that $\Delta_t(P)$ has $r$ facets. Let $P_1,\ldots, P_m$ denote the connected components of the Hasse diagram of $P$. By Remark \ref{maxmin} we may assume that there exists a path $x_{i_1},\ldots,x_{i_t}$ such that $x_{i_t}$ is a maximal element belonging to a unique maximal chain. Without loss of generality, we may assume that $x_{i_1},\ldots,x_{i_t}$ is contained in $P_1$. We conclude by Lemma~\ref{lem:leaf} that  $\Delta_t(P_1)$ has a leaf. Obviously, for any $F\in\mathcal{F}(\Delta_t(P))\setminus\mathcal{F}(\Delta_t(P_1))$, $F\cap\{x_{i_1},\ldots,x_{i_t}\}=\emptyset$. Hence $\{x_{i_1},\ldots,x_{i_t}\}$ is leaf of $\Delta_t(P)$.
 
  Let $\{G_1,\ldots,G_k\}=\mathcal{F}(\Delta_t(P))\setminus\{\{x_{i_1},\ldots,x_{i_t}\}\}$ and $Q$ the restriction of the poset $P$ to the set of elements from $G_1,\ldots,G_k$. Obviously, the Hasse diagram of $Q$ is a forest as an undirected graph and, by the induction hypotheses, $\Delta_t(Q)$ is a simplicial forest, hence it has a leaf. Thus $\Delta_t(P)$ is a simplicial forest.
 \end{proof}
 
Applying Theorem \ref{simpforest} to Theorem \ref{thm:simpForest} we obtain the following generalization of the result of He and Van Tuyl \cite{HT} concerning rooted trees.
 
 \begin{corollary}\label{cor:seqCM}
 Let $P$ be a forest poset and let $t\geq2$. Then $S/I_{\Delta_t(P)}$ is sequentially Cohen-Macaulay.
 \end{corollary}
 
 \subsection{Chain posets}
 In this section, we investigate the particular case of a chain on $n$ elements. Since those posets are trees, we already know by the results of the last section that their path ideals are sequentially Cohen-Macaulay. An obvious question that remains open is under which assumptions those ideals are even Cohen-Macaulay (in the pure sense). We give a complete answer to this question. 
 
 In the following let $n$ be a positive integer, let $S=k[x_1,\ldots,x_n]$, where $k$ is an arbitrary field, and let $L_n$ denote a chain on $n$ elements.
 Our first aim is to determine the height of the path ideals of $L_n$. 
 
 \begin{proposition}\label{path} Let $2\leq t\leq n$ be integers. 
 Then $\height(I_t(L_n))=\lfloor\frac{n}{t}\rfloor$.
 \end{proposition}  
 \begin{proof} We set $N=\lfloor\frac{n}{t}\rfloor$. Since $x_1\cdots x_t,x_{t+1}\cdots x_{2t},\ldots,x_{(N-1)t+1}\cdots x_{Nt}$ are in the unique minimal set of generators of $I_t(L_n)$, we deduce that $\height(I_t(L_n))\geq N$. If we consider the monomial prime ideal $ \frak{p}=\langle x_t,x_{2t},\ldots,x_{Nt}\rangle$, then $ \frak{p}\supseteq I_t(L_n)$. The statement follows. 
 \end{proof}
 Using the above result, we can immediately determine the Krull dimension of $S/I_t(L_n)$.
 
 \begin{corollary} Let $2\leq t\leq n$ be integers. Then $\dim(S/I_t(L_n))=n-\lfloor\frac{n}{t}\rfloor$.
 \end{corollary}
 In \cite{HT} He and Van Tuyl computed the projective dimension of the path ideal of the chain $L_n$ (In their paper, they referred to $L_n$ as the \emph{line graph} on $n$ elements.). 
 \begin{theorem}\cite[Theorem 4.1]{HT}\label{pdline} Let $2\leq t\leq n$ be integers. Then
 \begin{align*}
\pd(S/I_t(L_n))=\begin{cases}
	\frac{2(n-d)}{t+1},& \mbox{ if } n\equiv d\ (\mod(t+1)), \mbox{ with } 0\leq d\leq t-1;\\
	\frac{2n-(t-1)}{t+1},&\mbox{ if }n\equiv t\ (\mod(t+1)).
  \end{cases}
 \end{align*}
 \end{theorem} 
 The above result combined with Proposition~\ref{path} enables us to characterize which path ideals of chains are Cohen-Macaulay.
 
  \begin{proposition} \label{prop:CMpath}
Let $2\leq t\leq n$ be integers. Then $I_t(L_n)$ is Cohen-Macaulay if and only if $n=2t$ or $n=t$.
 \end{proposition}
\begin{proof}One may note that, by the Auslander-Buchsbaum formula, $I_t(L_n)$ is Co\-hen-Ma\-cau\-lay if and only if $\pd(S/I_t(L_n))=\height(I_t(L_n))$. By taking into account Theorem \ref{pdline}, we split the proof into two cases:

\textit{Case I:} First assume $n\equiv t \mod(t+1)$, i.\,e., $n=s(t+1)+t$, for some $s$. In this case, one has 
$n=(s+1)t+s$ and, by Theorem \ref{pdline}, $\pd(S/I_t(L_n))=(2n-(t-1))/(t+1)=2s+1$. 

Let $s=dt+r$, for some $0\leq r<t$, i.e., $n=(s+d+1)t+r$. In this case, $\height(I_t(L_n))=s+d+1$ (see Proposition \ref{path}) and  $I_t(L_n)$ is Cohen-Macaulay if and only if $s+d+1=2s+1$, which implies $s=d$. In particular, we get that $s(1-t)=r$, which holds only if $r=0$, hence $s=d=0$. Therefore we must have $n=t$.

\textit{Case II:} Assume now that $n\equiv d \mod(t+1)$ for some $0\leq d\leq t-1$, i.e, $n=s(t+1)+d$ for some $s$. One has that, 
in this case, $n=st+s+d$ and, by Theorem \ref{pdline}, $\pd(S/I_t(L_n))=2(n-d)/(t+1)=2s$. 

Let $s+d=d_1t+r$ for some $0\leq r<t$, i.e., $n=(s+d_1)t+r$. In this case, $\height(I_t(L_n))=s+d_1$ (see Proposition \ref{path}) and $I_t(L_n)$ is Cohen-Macaulay if and only 
if $s+d_1=2s$, which implies that $s=d_1$. We get that $d_1+d=d_1t+r$, i.e., $d=d_1(t-1)+r$. Since $d\leq t-1$, the only possible cases are $d_1=1$ 
and $r=0$ or $d_1=0$. In the first case, we get $s=1$, which gives us $n=2t$. In the second case, we obtain $s=0$. Hence, $n=r<t$, which is a contradiction since $n\geq t$ by our assumption.
\end{proof}

\subsection{Cycle posets}
In this section, we study posets, whose Hasse diagram --~viewed as an undirected graph~-- is a cycle. We will refer to those posets as \emph{cycle posets}.  
Observe that in each cycle poset $P$ there exist two unique maximal paths $x_{i_1},\ldots,x_{i_m}$ and $x_{j_1},\ldots,x_{j_r}$ with $x_{i_1}=x_{j_1}$ and $x_{i_m}=x_{j_r}$. 
As, clearly, a cycle poset is determined by the lengths of those two paths, we use $C_{m,r}$ to denote a cycle poset with unique maximal paths of lengths $m$ and $r$. 
Obviously, $C_{m,r}=C_{r,m}$. In the proofs of this section, we will always denote the elements of the two chains of $C_{m,r}$ by $x_{i_1},\ldots,x_{i_m}$ and $x_{j_1},\ldots,x_{j_r}$ 
without stating this again. 

As for tree and line posets, we ask the question under which conditions cycle posets have (sequentially) Cohen-Macaulay path ideals. We give a partial answer to this question by characterizing those cycle posets $C_{m,r}$ for which $\Delta_{t}(C_{m,r})$ is a simplicial tree. 
For cycles (in the sense of graphs), the corresponding question has been anwered in \cite{SKT}.

\begin{theorem}\label{thm:cycle}
 Let $m,r,t\geq 2$ be integers and let $C_{m,r}$ be a cycle poset. Then $\Delta_t(C_{m,r})$ is a simplicial tree if and only if one of the following holds:
\begin{itemize}
	\item[(i)] $C_{m,r}$ is not a graded poset and $t>\min(m,r)$.
	\item[(ii)] $C_{m,r}$ is a graded poset and $t=m$. 
\end{itemize}
\end{theorem}
\begin{proof}
 ``$\Rightarrow$'': First suppose that $C_{m,r}$ is not graded, that is $r\neq m$. Let $r=\min(m,r)$ and assume by 
contradiction that $\Delta_t(C_{m,r})$ is a  simplicial tree and $t\leq r$. We show that $\Delta_t(C_{m,r})$ has no leaves, 
which contradicts our assumption.

We split the proof in two cases:

\textit{Case I: $t<r$. } Let $F=\{x_{i_k},\ldots,x_{i_{k+t-1}}\}$ be a facet of $\Delta_t(C_{m,r})$ with $1<k\leq m-t$. One may note that 
such a facet exists since $t\leq m-2$.  Consider the facets $F_1=\{x_{i_{k+1}},\ldots,x_{i_{k+t}}\}$ and 
$F_2=\{x_{i_{k-1}},\ldots,x_{i_{k+t-2}}\}$. We have that $F_1\cap F=F\setminus\{x_{i_k}\}$ and $F_2\cap F=F\setminus\{x_{i_{k+t-1}}\}$. 
If $F\neq G\in\mathcal{F}(\Delta_t(P))\setminus\{F_1,F_2\}$, then, obviously, $F_1\cap F\nsubseteq G$ or $F_2\cap F\nsubseteq G$ since, otherwise, 
$F\subseteq G$. This shows that $F$ cannot be a leaf. 
One may argue in the same way for the facets of the form $F=\{x_{j_k},\ldots,x_{j_{k+t-1}}\}$ with $1<k\leq r-t$, if they exist.

We consider now the case that $F=\{x_{i_1},\ldots,x_{i_t}\}$. Then, for the facets $F_1=\{x_{i_{2}},\ldots,x_{i_{t+1}}\}$ and 
$F_2=\{x_{j_{1}},\ldots,x_{j_{t}}\}$, $x_{i_1}=x_{j_1}$ it holds that $F_1\cap F=F\setminus\{x_{i_1}\}$ and $F_2\cap F=\{x_{i_1}\}$. If 
$F\neq G\in\mathcal{F}(\Delta_t(P))\setminus\{F_1,F_2\}$, then obviously $F_1\cap F\nsubseteq G$ or $F_2\cap F\nsubseteq G$ since, otherwise, 
$F\subseteq G$. Therefore $F$ cannot be a leaf. One may argue in the same way for the facets of the form $F=\{x_{i_{m-t+1}},\ldots,x_{i_{m}}\}$,
 $F=\{x_{j_{r-t+1}},\ldots,x_{j_{r}}\}$, and $F=\{x_{j_{1}},\ldots,x_{j_{t}}\}$.

\textit{Case II: $t=r$.} 
Let us consider the facet $F=\{x_{j_1},\ldots,x_{j_r}\}$. Intersecting $F$ with $F_1=\{x_{i_{m-t+1}},\ldots,x_{i_{m}}\}$ and 
$F_2=\{x_{i_{1}},\ldots,x_{i_{t}}\}$, $x_{i_1}=x_{j_1}$ and $x_{i_m}=x_{j_r}$ we obtain $F_1\cap F=\{x_{i_1}\}$ and $F_2\cap F=\{x_{j_r}\}$. 
However, for all other facets $F\neq G\in\mathcal{F}(\Delta_t(P))\setminus\{F_1,F_2\}$, we have $G\cap F=\emptyset$. Hence, $F$ cannot be a leaf. 
The proof for facets of the form $\{x_{i_k},\ldots,x_{i_{k+t-1}}\}$ with $1\leq k\leq m-t$ works as in Case 1.

Assume  now that $C_{m,r}$ is a graded poset, i.e., $r=m$. The proof works as in Case I, once one assumes by contradiction that $\Delta_t(C_{m,r})$ 
is a simplicial tree and $t<m$.

``$\Leftarrow$'': (i) We assume that $\min(m,r)=r$. Since $t>r$, one may note that $\Delta_t(C_{m,r})$ is in fact the simplicial complex which arises 
from the path $x_{i_1},\ldots,x_{i_m}$. Therefore it is a simplicial tree, by \cite[Theorem 2.7]{HT}.

(ii) Since $t=r=m$, one may note that $\Delta_t(C_{m,r})$ is the simplicial complex whose facets are $\{x_{i_1},\ldots,x_{i_m}\}$ 
and $\{x_{j_1},\ldots,x_{j_r}\}$. Obviously, this complex is a simplicial tree.
\end{proof}

Using Theorem \ref{simpforest} we obtain the following class of cycle posets for which the path ideals are sequentially Cohen-Macaulay:

\begin{corollary} \label{cor:cycleCM}
Let $m,r,t\geq 2$ be integers and let $C_{m,r}$ be a cycle poset. Then $C_{m,r}$ is sequentially Cohen-Macaulay, if one of the following conditions holds:
\begin{itemize}
	\item[(i)] $C_{m,r}$ is not a graded poset and $t>\min(r,m)$;
	\item[(ii)] $C_{m,r}$ is a graded poset and $t=m$.
\end{itemize}
Moreover, in the first case, $S/I_t(C_{m,r})$ is Cohen-Macaulay if and only if $t=\max(r,m)$ or $2t=\max(r,m)$. In the second 
case, $S/I_t(C_{m,r})$ is Cohen-Macaulay if and only if $t=2$, that is $C_{m,r}$ is the chain poset $\{x_{i_1}<x_{i_2}\}$.
\end{corollary}
\begin{proof}
 It only remains to show the last statement. Assume that $C_{m,r}$ is not graded and that $t>\min(r,m)$. Without loss of generality, let $r=\min(r,m)$. 
Then $I_t(C_{m,r})$ is the path ideal of the line graph $L_m$ considered as an ideal in $k[x_{i_1},\ldots,x_{i_m},x_{j_1},\ldots,x_{j_r}]$. 
From Proposition~\ref{prop:CMpath} we infer that $I_t(C_{m,r})$ is Cohen-Macaulay if and only if $t=m$ or $2t=m$.

Assume now that $C_{m,r}$ is graded. If $t=m=2$, then again Proposition \ref{prop:CMpath} implies that $I_t(C_{m,r})$ is Cohen-Macaulay. 
If we have $t=m$ and $t\neq 2$, then, $F=\{x_{i_2},\ldots,x_{i_{t-1}},x_{i_t},x_{j_2},\ldots,x_{j_{t-1}}\}$ and 
$G=\{x_{i_1},x_{i_3},\ldots,x_{i_t},x_{j_3},\ldots,x_{j_{t-1}}\}$ are both facets of the Stanley-Reisner complex $\Delta_{t,C_{m,r}}$ of $I_t(C_{m,r})$. 
Since $|F|=2t-3>2t-4=|G|$, this complex is not pure and hence, $I_t(C_{m,r})=I_{\Delta_{t,C_{m,r}}}$ is not Cohen-Macaulay.
\end{proof}

As we did for chain posets, we are also able to compute the height of path ideals of cycle posets.
\begin{proposition} \label{prop:height}
Let $m,r,t\geq 2$ be integers and let $C_{m,r}$ be a cycle poset.Then
	\begin{align*}
\height(I_t(C_{m,r}))=\begin{cases}
	\frac{m+r}{t}-1,& \mbox{ if } t \mbox{ divides } m \mbox{ and } r;\\
	\lfloor\frac{m}{t}\rfloor+\lfloor\frac{r}{t}\rfloor,&\mbox{ otherwise}.
	 \end{cases}
	\end{align*}
\end{proposition}
\begin{proof} 
First assume that $t$ divides both $m$ and $r$, i.e., $m=Mt$ and $r=Rt$ for positive integers $M,R$. 
Since  $x_{i_2}\cdots x_{i_{t+1}},x_{i_{t+2}}\cdots x_{i_{2t+1}}$,$\ldots,x_{i_{(M-2)t+2}}\cdots x_{i_{(M-1)t+1}}$, 
$x_{j_1}\cdots x_{j_t},x_{j_{t+1}}\cdots x_{j_{2t}},\ldots,$ $x_{j_{(R-1)t+1}}\cdots x_{j_{Rt}}$ are monomials in $I_{t}(C_{m,r})$ which have disjoint 
supports, one has that $\height(I_t(C_{m,r}))\geq M+R-1$. 
On the other hand, the monomial prime ideal $\frak{p}=\langle x_{i_2},x_{i_{t+2}},\ldots,x_{i_{(M-2)t+2}},x_{j_t},x_{j_{2t}},\ldots,x_{j_{Rt}}\rangle$ 
contains $I_t(C_{m,r})$ and obviously $\height(\frak{p})=M+R-1$. Hence, $\height(I_t(C_{m,r}))= M+R-1$.

Now assume that at least one of $m$ and $r$ is not divisible by $t$. 
For simplicity, we set $M=\lfloor\frac{m}{t}\rfloor$ and $R=\lfloor\frac{r}{t}\rfloor$. 
We split the proof in two cases:

\textit{Case I: }\rm Suppose that $t$ divides $m$ but does not divide $r$, i.e., $m=Mt$ and $r> Rt$.
If $r> Rt+1$, then the monomials $x_{i_1}\cdots x_{i_{t}},\ldots,$ $x_{i_{(M-1)t}}\cdots x_{i_{Mt}}$, $x_{j_2}\cdots x_{j_{t+1}},\ldots,$ 
$x_{j_{(R-1)t+2}}\cdots x_{j_{Rt+1}}$ are in $I_{t}(C_{m,r})$ and have disjoint supports. We get that $\height(I_t(C_{m,r}))\geq M+R$. 
On the other hand, $I_t(C_{m,r})$ is contained in the prime ideal 
$\frak{p}=\langle x_{i_t},x_{i_{2t}},\ldots,x_{i_{Mt}},x_{j_2},x_{j_{t+2}},\ldots,x_{j_{(R-1)t+2}}\rangle$. Since $\height(\frak{p})=M+R$, the claim follows.
 
Let $r= Rt+1$. Let $\frak{p}$ be a minimal prime ideal with $\frak{p}\supseteq I_t(C_{m,r})$. We claim that $\height(\frak{p})\geq M+R$. 
We distinguish three cases.

\textit{Case I (a): $x_{i_{Mt}}\notin \frak{p}$. }\rm 
Let $I_1$ and $I_2$ be the path ideals of length $t$ of the chain posets $x_{i_1},\ldots,x_{i_{Mt}}$ and 
$x_{j_2},\ldots,x_{j_{Rt+1}}$, respectively. We clearly have $\frak{p}\supseteq I_1$ and $\frak{p}\supseteq I_2$, which, by Proposition~\ref{path}, implies that 
$|\frak{p}\cap\{x_{i_1},\ldots,x_{i_{Mt}}\}|\geq M$ and $|\frak{p}\cap\{x_{j_2},\ldots,x_{j_{Rt+1}}\}|\geq R$. Since, by assumption, $x_{i_{Mt}}$ does not 
belong to $\frak{p}$, we can conclude $\height(\frak{p})\geq M+R$. 

\textit{Case I (b): $x_{i_1}\notin \frak{p}$. }\rm 
The claim follows as in Case I (a) if one considers the path ideal of length $t$ of the chain poset $x_{j_1},\ldots,x_{j_{Rt}}$.

\textit{Case I (c): $x_{i_1}\in \frak{p}$, $x_{i_{Mt}}\in \frak{p}$. }\rm
Let $J_1$ and $J_2$ be the path ideals of length $t$ of the chain posets $x_{i_2},\ldots,x_{i_{Mt-1}}$ and $x_{j_2},\ldots,x_{j_{Rt}}$, respectively. 
Since $t\geq 2$, Proposition~\ref{path} implies that $\height(J_1)=M-1$ and $\height(J_2)=R-1$. This, in particular, implies 
$|\frak{p}\cap \{x_{i_2},\ldots,x_{i_{Mt-1}}|\geq M-1$ and $|\frak{p}\cap \{x_{j_2},\ldots,x_{j_{Rt}}|\geq R-1$. Since $x_{i_1},x_{i_{Mt}}\in\frak{p}$ 
by assumption, we infer that $\height(\frak{p})\geq (M-1)+(R-1)+2=M+R$. 

It follows that $\height(I_t(C_{m,r}))\geq M+R$. 
Moreover, as $I_t(C_{m,r})$ is contained in the prime ideal 
$\frak{p}=\langle x_{i_{t}},x_{i_{2t}},\ldots,x_{i_{Mt}},x_{j_1},x_{j_{t+1}},\ldots,x_{j_{(R-1)t+1}}\rangle$ and as $\height(\frak{p})=M+R$, we conclude $\height(I_t(C_{m,r}))=M+R$.
 
The case, that $r$ is divisible by $t$ but $m$ is not, follows by an analog reasoning.
 
 \textit{Case II: }\rm Suppose that none of $m$ and $r$ is divisible by $t$. Since the monomials 
$x_{i_2}\cdots x_{i_{t+1}},\ldots,$ $x_{i_{(M-1)t+2}}\cdots x_{i_{Mt+1}}$, $x_{j_1}\cdots x_{j_{t}},\ldots,$ $x_{j_{(R-1)t+1}}\cdots x_{j_{Rt}}$ 
lie in $I_{t}(C_{m,r})$ and have disjoint supports, we obtain $\height(I_t(C_n))\geq M+R$. Moreover, the prime ideal 
$\frak{p}=\langle x_{i_{t}},x_{i_{2t}},\ldots,$ $x_{i_{Mt}},x_{j_t},x_{j_{2t}},\ldots,x_{j_{Rt}}\rangle$ contains $I_t(C_{m,r})$ and 
$\height(\frak{p})=M+R$. Hence, $\height(I_t(C_{m,r}))=M+R$. 
\end{proof}

\section{Defining ideals of the Luce-decomposable model}\label{sect:specialClasses}
Given a graded poset $P$ on a finite ground set $\Omega$, we consider the set $\M(P)$ of maximal saturated chains in $P$. The \emph{Luce-decomposable model}, also referred to as \emph{Csisz\'ar model},
is a specific toric model --~arising in algebraic statistics~--, whose state space is the set of maximal saturated chains of a certain poset $P$. The probability $p_{\pi}$ 
of observing the maximal saturated chain $\pi=y_1<\ldots <y_r$ is given by the monomial map 
\begin{equation*}
\Phi_P:\; p_{\pi}\mapsto d_{y_1<y_2}\cdot d_{y_2<y_3}\cdots d_{y_{r-1}<y_r},
\end{equation*}
where for each cover relation $a<b$ in $P$ we have a model parameter $d_{a<b}$. The Luce-decomposable model is a widely studied model in statistics e.\,g., 
it was studied by Csisz\'ar for the case of the Boolean lattice $2^{[n]}$ \cite{Cs2,Cs1}. More recently, in \cite{SW} Sturmfels and Welker described 
the model polytope of this model for a general poset by giving a set of linear inequalities. They also constructed a Gr\"obner and a Markov basis for the toric ideal
of this model, which is defined as the kernel of the map $\Phi_P$. \\
In this article, we are interested in properties of the ideal 
\begin{equation*}
I_{\LD}(P):=\langle \Phi_P(\pi)~:~\pi\in \M(P)\rangle,
\end{equation*}
 which will be referred to as the 
\emph{$\LD$-ideal} of $P$. 
Note, that, if $P$ is graded, then $I_{\LD}(P)$ can be viewed as the path ideal of a certain poset $\hat{P}$. Indeed, 
let $(\hat{P},<_{\hat{P}})$ be the poset, whose elements correspond to the cover relations of $P$ and with the ordering $<_{\hat{P}}$ defined in the 
following way: For two cover relations $(a<_P b)$ and $(c<_P d)$ in $P$, we have $(a<_P b)<_{\hat{P}}(c<_P d)$ if and only if $b<_P c$. 
Since $P$ is graded, so is $\hat{P}$ and if $\hat{P}$ is of rank $r-1$, then $I_r(\hat{P})=I_{\LD}(P)$. 
We want to emphasize, that from now on --~if not stated otherwise~-- by a chain in $P$ we rather mean a sequence of edges, than a sequence of nodes in the Hasse 
diagram of $P$. Clearly, a chain can be described in both ways and from a practical point of view, the only difference is that now we will not consider 
a labeling of the elements of $P$ but label the edges, i.e., the cover relations, of the Hasse diagram of $P$. 
In this section, we focus on the investigation of algebraic properties of $\LD$-ideals and also of the corresponding path ideals for special classes of posets. 


\subsection{Primary decomposition and height}
We first describe the primary components of $I_{\LD}(P)$ for a graded poset $P$. 
We assume that the cover relations of $P$ are labeled with $x_1,\ldots,x_m$ for a certain $m$ and consider $I_{\LD}(P)$ as an ideal 
in the polynomial ring $k[x_1,\ldots,x_m]$. 

We call a set $T$ of edges 
of the Hasse diagram of $P$ a \emph{minimal cut-set} of $P$ if $T\cap C \neq \emptyset$ for every maximal chain $C$ in $P$. 
It is well-known, see e.\,g., \cite[Remark 2]{Fa} that the minimal primary components of $I_{\LD}(P)$ correspond to the minimal vertex covers of its facet 
complex.  
Since those correspond to minimal cut-sets of $P$, we obtain the following result.
 
\begin{proposition}\label{prop:primary}
Let $(P,<_P)$ be a poset. Then the standard primary decomposition of $I_{\LD}(P)$ is given by
\begin{equation*}
I_{\LD}(P)=\bigcap_{T \mbox{ minimal cut-set of }P}P_T,
\end{equation*}
where $P_T=\langle x_i~:x_i\in T\rangle$.
\end{proposition}

From this, one directly obtains the following characterization for the height of the considered ideals.

\begin{corollary}
Let $(P,<_P)$ be a poset. Then the height of $I_{\LD}(P)$ equals the minimal cardinality of a cut-set of $P$.
\end{corollary}

\subsection{Diamond posets}\label{subsect:diamond}

In this section, we study the family of \emph{diamond posets}. Roughly speaking, those are posets, whose Hasse diagrams
have the shape of piled diamonds. 
More formally, they can be constructed in the following way. For $n=1$, we let $D_1$ to be the poset on $\{x_1,x_2,x_3,x_4\}$ with cover relations 
$x_1 <_{D_1} x_2$, $x_1<_{D_1} x_3$, $x_2<_{D_1} x_4$ and $x_3<_{D_1} x_4$. For $n\geq 2$, we set $D_n=\widetilde{D}_{n-1}\oplus D_1$, where 
$\widetilde{D}_{n-1}$ is obtained from $D_{n-1}$ by removing its maximum element. 
The poset $D_n$ is called the \emph{$n$\textsuperscript{th} diamond poset}. Figure \ref{diamond3} shows the Hasse diagram of the $3$\textsuperscript{rd} diamond poset. 
In the following, we denote the elements and cover relations 
of $D_n$ with $x_1,\ldots,x_{3n+1}$ and $y_1,\ldots,y_{4n}$, respectively. 
\begin{figure}[h]
\begin{center}
\includegraphics[width=3.5cm]{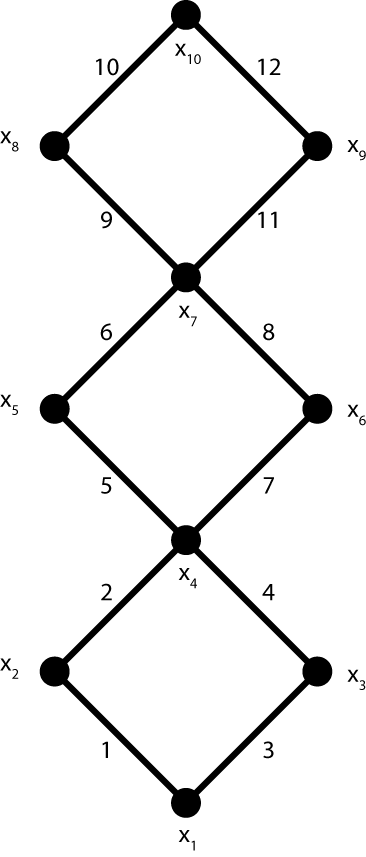}
\end{center}
\caption{The Hasse diagram of $D_3$ including a specific vertex and edge labeling.}
\label{diamond3}
\end{figure}
 
For fixed $n$, let $S_n=k[y_1,y_2,\ldots,y_{4n}]$ and $T_n=k[x_1,\ldots,x_{3n+1}]$, where $k$ is an arbitrary field. Then $I_{\LD}(D_n)$ and 
$I_{2n+1}(D_n)$ are ideals in $S_n$ and $T_n$, respectively. 

Our main result of this section provides a formula for the $\mathbb{Z}$-graded Betti numbers of $S_n/I_{\LD}(D_n)$ and $T_n/I_{2n+1}(D_n)$.

\begin{theorem}\label{thm:diamond}
Let $n\geq 1$ be a positive integer. Then
$\beta_{0,0}(S_n/I_{\LD}(D_n))=1$ and 
\begin{align*}
\beta_{i,j}(S_n/I_{\LD}(D_n))=
\begin{cases}
\binom{n}{i-1}2^{n-i+1}, \qquad&\mbox{ if } j=2i+2n-2\\
0,\qquad&\mbox{ otherwise}
\end{cases}
\end{align*}
for $1\leq i\leq n+1$. Moreover, $\beta_{i,j}(S_n/I_{\LD}(D_n))=0$ for $i>n+1$. 
Similarly, $\beta_{0,0}(T_n/I_{2n+1}(D_n))=1$ and 
\begin{align*}
\beta_{i,j}(T_n/I_{2n+1}(D_n))=
\begin{cases}
\binom{n}{i-1}2^{n-i+1}, \qquad&\mbox{ if } j=i+2n\\
0,\qquad&\mbox{ otherwise}
\end{cases}
\end{align*}
for $1\leq i\leq n+1$. Moreover, $\beta_{i,j}(T_n/I_{2n+1}(D_n))=0$ for $i>n+1$.
 In particular, both, $S_n/I_{\LD}(D_n)$ and $T_n/I_{2n+1}(D_n)$ have a pure resolution. 
\end{theorem}

We remark that the statements for the path ideals of $D_n$ will directly follow from the statements for the $\LD$-ideals by using a specific chain isomorphism between the 
minimal resolutions.  

The proof requires a number of auxiliary results.  
First, we introduce a specific edge labeling of the Hasse diagram of $D_n$. 
Let $a^{(i)}$ respectively $d^{(i)}$ be the minimal respectively the maximal element of the $i$\textsuperscript{th} diamond, counted from the 
bottom, and let $b^{(i)}$ and $c^{(i)}$ such that $a^{(i)}<_{D_n} b^{(i)},c^{(i)}<_{D_n}d^{(i)}$. Then the edges corresponding to 
the cover relations $a^{(i)}<_{D_n} b^{(i)}$ and $b^{(i)}<_{D_n} d^{(i)}$ are labeled with $4(i-1)+1$ and $4(i-1)+2$, respectively, 
whereas those corresponding to the cover relations $a^{(i)}<_{D_n} c^{(i)}$ and $c^{(i)}<_{D_n}d^{(i)}$ receive the labels $4(i-1)+3$ and $4i$, 
respectively. (See  Figure \ref{diamond3} for the Hasse diagram of $D_3$ with this labeling.) For $1\leq i\leq n$, let $L_i^{(1)}=\{4(i-1)+1,4(i-1)+2\}$ and $L_i^{(2)}=\{4(i-1)+3,4(i-1)+4\}$. In the following, we will identify 
an edge in the Hasse diagram of $D_n$ with its label.

Let $P$ be a subposet of $D_n$ of maximal rank. We distinguish two types of those subposets: 
\begin{itemize}
\item[type I:] $P=\bigcup_{i=1}^n L_i$, where $L_i\in \{L_i^{(1)},L_i^{(2)}\}$.
\item[type II:] $P=\bigcup_{i=1}^n L_i$ and there exists $1\leq s\leq n$ and $j_1<\cdots <j_s$ such that $L_{j_t}=L_{j_t}^{(1)}\cup L_{j_t}^{(2)}$ 
for $1\leq t\leq s$ and $L_i\in \{L_i^{(1)},L_i^{(2)}\}$ for $i\in [n]\setminus\{j_1,\ldots,j_s\}$. 
\end{itemize}
Note that a subposet of type I is just a maximal chain in $D_n$. A subposet of type II consists of a maximal chain in $D_n$, where $s$ diamonds have been 
completed by the so far missing edges. Figure \ref{figure:types} shows subposets of type I and II.
\begin{figure}
        \centering
        \begin{subfigure}[b]{0.23\textwidth}
                \centering
                \includegraphics[width=1.2cm]{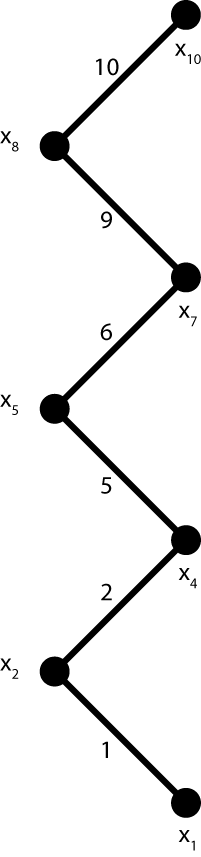}
                \caption{}
\label{typeIa}
        \end{subfigure}%
        \begin{subfigure}[b]{0.23\textwidth}
                \centering
                \includegraphics[width=2.2cm]{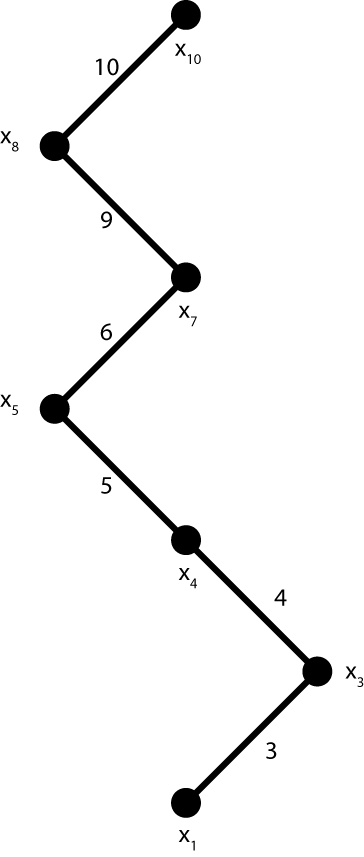}
                \caption{}
\label{typeIb}
        \end{subfigure}
        \begin{subfigure}[b]{0.23\textwidth}
                \centering
                \includegraphics[width=2.2cm]{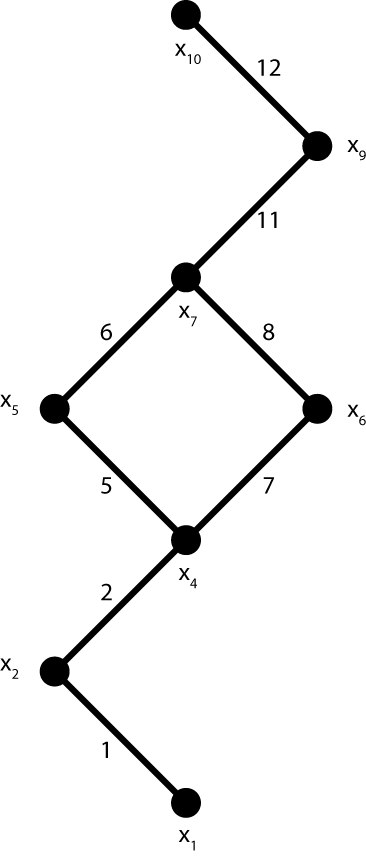}
                \caption{}
\label{typeIIa}
        \end{subfigure}
\begin{subfigure}[b]{0.23\textwidth}
                \centering
                \includegraphics[width=2.2cm]{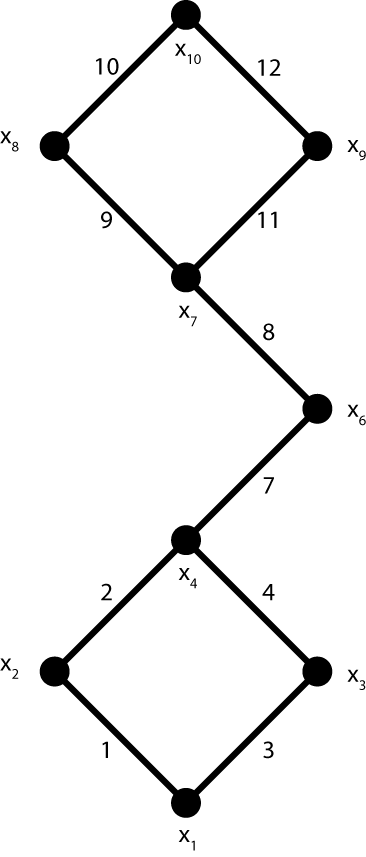}
                \caption{}
\label{typeIIb}
        \end{subfigure}
        \caption{Subposets of type I ((A), (B)) and type II ((C),(D))}
\label{figure:types}
\end{figure}

To keep our notations simple, we write $\Delta_{\LD,n}$ for the Stanley-Reisner complex of $I_{\LD}(D_n)$. If $P$ is a subposet of $D_n$, then 
$(\Delta_{\LD,n})_P$ will denote the restriction of $\Delta_{\LD,n}$ to the set of cover relations of $P$. 
Theorem \ref{thm:diamond} will follow from the following four lemmas, which are concerned with the homology of $\Delta_{\LD,n}$ and certain 
induced subcomplexes.

\begin{lemma}\label{lem:typeIIa}
For $n\geq 1$, it holds that
\begin{align*}
\widetilde{H}_i(\Delta_{\LD,n};k)=
\begin{cases}
k, &\mbox{ if } i=3n-2\\
0, &\mbox{ otherwise.}
\end{cases}
\end{align*}
\end{lemma}

\begin{proof}
%
We proceed by induction on $n$. If $n=1$, then it is easy to see that $\Delta_{\LD,1}$ is a $4$-gon. Hence, $\widetilde{H}_1(\Delta_{\LD,1};k)=k$ 
and $\widetilde{H}_i(\Delta_{\LD,1};k)=0$ for $i\neq 1$. 

Assume that $n\geq 2$. A face $F$ of $\Delta_{\LD,n}$ corresponds to a set of edges of the Hasse diagram of $D_n$ that does not contain any 
maximal chain. This is the case if $F$ does not contain any maximal chain of the restriction of $D_n$ to the first $n-1$ diamonds, 
or if $F$ does not contain any maximal chain of the $n$\textsuperscript{th} diamond. 
We conclude that 
\begin{equation*}
\Delta_{\LD,n}\cong \Delta_{\LD,n-1}\ast 2^{\{4n-3,4n-2,4n-1,4n\}}\cup 2^{[4n-4]}\ast \Delta_{\LD,1}.
\end{equation*}
Moreover, 
\begin{equation*}
\Delta_{\LD,n-1}\ast 2^{\{4n-3,4n-2,4n-1,4n\}}\cap 2^{[4n-4]}\ast \Delta_{\LD,1}=\Delta_{\LD,n-1}\ast\Delta_{\LD,1}.
\end{equation*}
Since $\Delta_{\LD,n-1}\ast 2^{\{4n-3,4n-2,4n-1,4n\}}$ and $2^{[4n-4]}\ast \Delta_{\LD,1}$ are both contractible, 
the Mayer-Vietoris sequence \eqref{eq:MV} for $(\Delta_{\LD,n-1}\ast 2^{\{4n-3,4n-2,4n-1,4n\}},2^{[4n-4]}\ast \Delta_{\LD,1})$ implies that
\begin{equation*}
\widetilde{H}_i(\Delta_{\LD,n};k)\cong \widetilde{H}_{i-1}(\Delta_{\LD,n-1}\ast\Delta_{\LD,1};k).
\end{equation*}
By the induction hypothesis, we have $\widetilde{H}_i(\Delta_{\LD,n-1};k)\neq 0$ if and only if $i=3n-5$. Hence, it follows from Theorem \ref{thm:join} that
\begin{align*}
\widetilde{H}_i(\Delta_{\LD,n};k)=
\begin{cases}
k, &\mbox{ if }  i=(3n-5)+1+2=3n-2\\
0, &\mbox{ otherwise.}
\end{cases}
\end{align*}
\end{proof}

\begin{lemma}\label{lem:typeI}
If $P$ is a subposet of $D_n$ of type I, then 
\begin{align*}
\widetilde{H}_i((\Delta_{\LD,n})_P;k)=
\begin{cases}
k, &\mbox{ if } i=2n-2\\
0, &\mbox{ otherwise.}
\end{cases}
\end{align*}
\end{lemma}

\begin{proof}
Faces of $\Delta_{\LD,n}$ are those subsets of the edge set of the Hasse diagram not containing any maximal chain of $D_n$. Thus, if 
we restrict $\Delta_{\LD,n}$ to a maximal chain $P$ of $D_n$, then it holds that $(\Delta_{\LD,n})_P=\{F~:~F\subsetneq P\}$ i.\,e., 
$(\Delta_{\LD,n})_P$ is the boundary of a $(|P|-1)$-simplex. Since $|P|=2n$ for a subposet $P$ of type I, this shows the claim.
\end{proof}

\begin{lemma}\label{lem:typeII}
Let $P$ is a subposet of $D_n$ of type II. If $s$ is the number of full diamonds in $P$, then 
\begin{align*}
\widetilde{H}_i((\Delta_{\LD,n})_P;k)=
\begin{cases}
k, &\mbox{ if } i=2n+s-2\\
0, &\mbox{ otherwise.}
\end{cases}
\end{align*}
\end{lemma}

\begin{proof}
Without loss of generality, we may assume that $P$ consists of the first $s$ diamonds of $D_n$ and the edges in $L_i^{(1)}$ for the remaining
$n-s$ diamonds, i.\,e., 
\begin{equation}\label{Type2}
P=\bigcup_{i=1}^s \{4(i-1)+1,4(i-1)+2,4(i-1)+3,4(i-1)+4\}\cup \bigcup_{i=s+1}^n L_i^{(1)}.
\end{equation}
For $s=n$, the claim follows from Lemma \ref{lem:typeIIa}.

Suppose that $1\leq s<n$. Using similar arguments as in the case $s=n$, we obtain the decomposition
\begin{equation}\label{eq:int1}
(\Delta_{\LD,n})_P= \Delta_{\LD,s}\ast 2^{\{4m+1,4m+2~:~s\leq m\leq n-1\}}\cup 2^{[4s]}\ast \partial\left(2^{\{4m+1,4m+2~:~s\leq m\leq n-1\}}\right).
\end{equation}
Moreover,
 \begin{align*}
&\Delta_{\LD,s}\ast 2^{\{4m+1,4m+2~:~s\leq m\leq n-1\}}\cap 2^{[4s]}\ast \partial\left(2^{\{4m+1,4m+2~:~s\leq m\leq n-1\}}\right)\\
=&\Delta_{\LD,s}\ast \partial\left(2^{\{4m+1,4m+2~:~s\leq m\leq n-1\}}\right).
\end{align*}
Since both complexes on the right-hand side of \eqref{eq:int1} are contractible 
the Mayer-Vietoris sequence \eqref{eq:MV} for those two complexes 
implies that
\begin{equation*}
\widetilde{H}_i((\Delta_{\LD,n})_P;k)= \widetilde{H}_{i-1}(\Delta_{\LD,s}\ast \partial\left(2^{\{4m+1,4m+2~|~s\leq m\leq n-1\}}\right);k).
\end{equation*}
As in the case for $\Delta_{\LD,n}$, it follows from Theorem \ref{thm:join} that 
\begin{align*}
&\empty\widetilde{H}_i((\Delta_{\LD,n})_P;k)=\\
=&\bigoplus_{j+l=i-2}\widetilde{H}_j(\Delta_{\LD,s};k)\otimes_k \widetilde{H}_l(\partial(2^{\{4m+1,4m+2~:~s\leq m\leq n-1\}});k)\\
=&
\small{\begin{cases}
\widetilde{H}_{3s-2}(\Delta_{\LD,s};k)\otimes_k \widetilde{H}_{2n-2s-2}(\partial(2^{\{4m+1,4m+2~|~s\leq m\leq n-1\}});k), &\mbox{ if } i-1=2n+s-3\\
0, &\mbox{ otherwise}\\
\end{cases}}\\
=&\begin{cases}
k, \mbox{ if } i=2n+s-2\\
0, \mbox{ otherwise}.
\end{cases}
\end{align*}
For the second equality we used Lemma \ref{lem:typeIIa} and that $\partial(2^{\{4m+1,4m+2~|~s\leq m\leq n-1\}})$ is the boundary of 
a $(2n-2s-1)$-simplex.
\end{proof}

\begin{lemma}\label{lem:typeIII}
If $P$ is a subposet of $D_n$, which is neither of type I nor of type II, then 
$\widetilde{H}_i((\Delta_{\LD,n})_P;k)=0$ for all $i$.
\end{lemma}

\begin{proof}
First assume that $P$ does not contain any maximal chain of $D_n$. Then $P$ does not contain any non-face of $\Delta_{\LD,n}$ and therefore  
$(\Delta_{\LD,n})_P$ has to be the full $(|P|-1)$-simplex on vertex set $P$. Thus, the claim follows in this case.
 
Next, assume that $P$ contains a maximal chain of $D_n$. Since $P$ is neither of type I nor of type II, there exists $1\leq t\leq n$ such that 
$|P\cap \{4(t-1)+1,4(t-1)+2,4(t-1)+3,4(t-1)+4\}|=3$. Without loss of generality, suppose that $4(t-1)+1\notin P$. Then there is no maximal chain 
$C$ of $D_n$ such that $C\subseteq P$ and $4(t-1)+2\in C$. This in particular means that for all faces $F \in(\Delta_{\LD,n})_P$ also 
$F\cup\{4(t-1)+2\}$ is a face of $(\Delta_{\LD,n})_P$. Hence $(\Delta_{\LD,n})_P$ is a cone with apex $4(t-1)+2$ and as such has trivial homology. 
\end{proof}

We can finally provide the proof of Theorem \ref{thm:diamond}.\\

{\sf Proof of Theorem \ref{thm:diamond}:}
We first show the claim for the $\LD$-ideals of $D_n$. 
We need to compute the Betti numbers of $S_n/I_{\LD}(D_n)$. Let $i\geq 1$ be fixed. By the Hochster formula (Theorem \ref{Hochster}) it holds that
\begin{align*}
\beta_{i,2i+2n-2}(S_n/I_{\LD}(D_n))=\sum_{\substack{P\subseteq [4n]\\|P|=2i+2n-2}}\dim_k\widetilde{H}_{2n+i-3}((\Delta_{\LD,n})_P;k).
\end{align*}
By Lemma \ref{lem:typeIII} the only subposets of $D_n$ with non-trivial contributions to the right-hand side of the above equation are those of 
type I or II. Moreover, Lemma \ref{lem:typeI} implies that those of type I can only contribute for $i=1$, whereas it follows from 
Lemma \ref{lem:typeII} that for $2\leq i\leq n+1$ exactly those subposets of type I that have $i-1$ full diamonds have to be considered. 
To summarize, we obtain (for $i=1$)
\begin{equation*}
\beta_{1,2n}(S_n/I_{D_n})=|\{P \mbox{ subposet of }D_n \mbox{ of type I}\}| 
\end{equation*}
and
\begin{align*}
&\beta_{i,2i+2n-2}(S_n/I_{D_n})\\
=&|\{P \mbox{ subposet of }D_n \mbox{ of type II containing } i-1 \mbox{ full diamonds}\}|
\end{align*}
for $2\leq i\leq n+1$. 
Since there are $2^n$ maximal chains in $D_n$, it follows that $\beta_{1,2n}=2^n$. \\
In order to construct a subposet of $D_n$ with exactly $i-1$ full diamonds, we first choose the $i-1$ full diamonds, which yields $\binom{n}{i-1}$ 
possibilities. For the remaining $n-i+1$ diamonds, we have to decide whether to take the set $L_u^{(1)}$ or to take the set $L_u^{(2)}$, 
which amounts to $2^{n-i+1}$ possibilities. Thus, in total, there are $\binom{n}{i-1}2^{n-i+1}$ subposets of $D_n$ with $i-1$ full diamonds, i.\,e., 
$\beta_{i,2i+2n-2}(S_n/I_{D_n})=\binom{n}{i-1}2^{n-i+1}$ for $1\leq i\leq n+1$. 
Since each non-trivial homology group of a restriction of $\Delta_{\LD,n}$ contributes to exactly one Betti number of $S_n/I_{D_n}$, 
we infer from Lemma \ref{lem:typeIII} and from the above reasoning that $\beta_{i,j}(S_n/I_{D_n})=0$ for $i>n+1$ or $1\leq i\leq n+1$ and 
$j\neq 2i+2n-2$.

It remains to show the statements for the path ideal $I_{2n+1}(D_n)$. Those can be deduced from the statements for the $\LD$-ideal $I_{\LD}(D_n)$ in the following way:
Define a map $\phi$ from the minimal set of generators of $I_{\LD}(D_n)$ to the minimal set of generators of $I_{2n+1}(D_n)$ by 
sending the generator of $I_{\LD}(D_n)$ corresponding to a certain chain to the generator of the same chain in $I_{2n+1}(D_n)$. This map gives an 
isomorphism between the minimal sets of generators of $I_{\LD}(D_n)$ and $I_{2n+1}(D_n)$. It is straightforward to show, that $\phi$ 
induces a chain isomorphism between the resolutions of these two ideals. It only remains to observe that in the $i$-th step of the resolutions, 
this chain isomorphism maps elements of degree $2i+2n-2$ to elements of degree $i+2n$. (This causes 
the different shifts in the resolutions.) 
\qed 

Using Theorem \ref{thm:diamond}, we can compute the projective dimension, the regularity and the depth of $S_n/I_{\LD}(D_n)$ and $T_n/I_{2n+1}(D_n)$.

\begin{corollary}\label{cor:diamond}
Let $n\geq 1$ be a positive integer. Then:
\begin{itemize}
\item[(i)] $\pd(S_n/I_{\LD}(D_n))=n+1$ and $\pd(T_n/I_{2n+1}(D_n))=n+1$,
\item[(ii)] $\reg(S_n/I_{\LD}(D_n))=3n-1$ and $\reg(T_n/I_{2n+1}(D_n))=2n$,
\item[(iii)] $\depth(S_n/I_{\LD}(D_n))=3n-1$ and $\depth(T_n/I_{2n+1}(D_n))=2n$.
\end{itemize}
\end{corollary}

\begin{proof}
We only prove the statements for the $\LD$-ideal $I_{\LD}(D_n)$ since those for $I_{2n+1}(D_n)$ follow by exactly the same arguments. 
Part (i) follows directly from Theorem \ref{thm:diamond}. For proving part (iii), one may note that, by the Auslander-Buchsbaum formula it holds that 
\begin{equation*}
\depth(S_n/I_{\LD}(D_n))=4n-\pd(S_n/I_{\LD}(D_n))=4n-(n+1)=3n-1.
\end{equation*}
For (ii), one has to observe that $2i+2n-2-i=2n+i-2$ is maximal for $i=n+1$. Therefore,
\begin{equation*}
\reg(S_n/I_{\LD}(D_n))=\max\{j-i~:~\beta_{i,j}(S_n/I_{\LD}(D_n))\neq 0\}=2n+(n+1)-2=3n-1.
\end{equation*}
\end{proof}

\begin{lemma}\label{lem:diamond}
Let $n\geq 1$ be a positive integer. Then: 
\begin{itemize}
\item[(i)] $\height(I_{\LD}(D_n)=2$ and $\height(I_{2n+1}(D_n))=1$, 
\item[(ii)] $\dim(S_n/I_{\LD}(D_n))=4n-2$ and $\dim(S_n/I_{2n+1}(D_n))=3n$.
\end{itemize}
\end{lemma}

\begin{proof}
By Proposition \ref{prop:primary} the  height of $I_{\LD}(D_n)$ equals the cardinality of a minimal cut-set of $D_n$. Clearly, 
the two edges in the Hasse diagram of $D_n$, labeled with $1$ and $3$ (see Figure \ref{diamond3}), form such a cut-set. 
Thus, $\height(I_{\LD}(D_n))=2$ and $\dim(S_n/I_{\LD}(D_n))=4n-2$.\\
For $T_n/I_{2n+1}(D_n)$, note that both, the minimal and the maximal element, of $D_n$ are a minimal vertex cover for $\Delta_{2n+1}(D_n)$. Hence,
by the discussion preceding Proposition \ref{prop:primary}, one has $\height(I_{2n+1}(D_n))=1$ and hence, $\dim(T_n/I_{2n+1}(D_n))=3n$.
\end{proof}

Corollary \ref{cor:diamond} and Lemma \ref{lem:diamond} enable us to characterize diamond posets whose $\LD$-ideal, respectively path ideal of paths of maximal length, is Cohen-Macaulay. 

\begin{corollary}
$S_n/I_{\LD}(D_n)$ is Cohen-Macaulay over $k$ if and only if $n=1$. There is no $n$ such that $T_n/I_{2n+1}(D_n)$ is Cohen-Macaulay over $k$.
\end{corollary}

\subsection{Products of chains}\label{subsect:chains}
In this section, we consider products of a $2$-element chain $L_2$ and an $n$-element chain $L_n$, where $n\geq 2$. To simplify notation, we set 
$L_{2,n}=L_2\times L_n$. Figure \ref{chains4} shows the Hasse diagram of the poset $L_{2,n}$. 
\begin{figure}[h]
\begin{center}
\includegraphics[width=6cm]{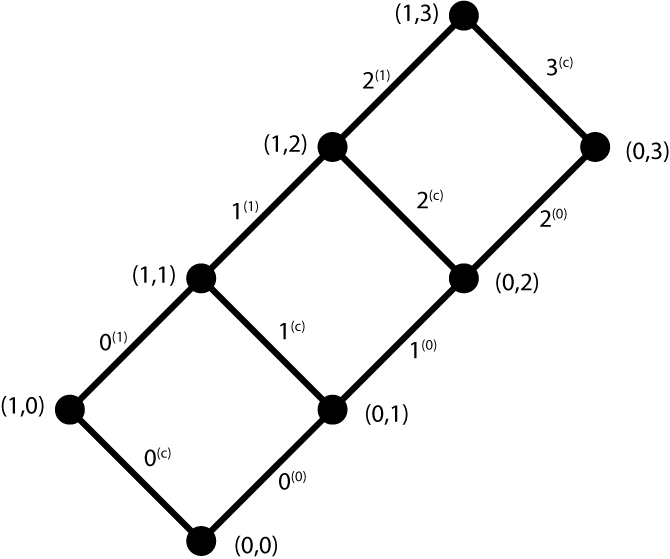}
\end{center}
\caption{The Hasse diagram of $L_{2,4}$ including a specific vertex and edge labeling.}
\label{chains4}
\end{figure}

We identify the elements of $L_{2,n}$ with $x_1,\ldots,x_{2n}$ and the edges in its Hasse diagram with $y_1,\ldots,y_{3n-2}$. 
 For a fixed $n$, let $S_n=k[y_1,\ldots,y_{3n-2}]$ and $T_n=k[x_1,\ldots,x_{2n}]$, 
where $k$ is an arbitrary field. Then $I_{\LD}(L_{2,n})$ and $I_{n+1}(L_{2,n})$ is an ideal in $S_n$ and $T_n$, respectively. Our main result of this 
section provides formulas for the $\mathbb{Z}$-graded Betti numbers of $S_n/I_{\LD}(L_{2,n})$ and $T_n/I_{n+1}(L_{2,n})$.

\begin{theorem}\label{thm:chains}
Let $n\geq 2$ be a positive integer. Then
 $\beta_{0,0}(S_n/I_{\LD}(L_{2,n}))=1$, $\beta_{1,n}(S_n/I_{\LD}(L_{2,n}))=n$ and
\begin{align*}
\beta_{i,j}(S_n/I_{\LD}(L_{2,n}))=
\begin{cases}
(2n-j+i-1)\binom{j-n-i}{i-2}, &\mbox{ for } n+2i-2\leq j\leq 2n+i-2\\
0,&\mbox{ otherwise}
\end{cases}
\end{align*}
for $2\leq i\leq n$. Moreover, $\beta_{i,j}(S_n/I_{\LD}(L_{2,n}))=0$ for $i>n+1$. In particular, the 
total Betti numbers are given by $\beta_i(S_n/I_{\LD}(L_{2,n}))=\binom{n}{i}$ for $0\leq i \leq n$.
Similarly, $\beta_{0,0}(T_n/I_{n+1}(L_{2,n}))=1$, $\beta_{1,n+1}(T_n/I_{n+1}(L_{2,n}))=n$ and $\beta_{2,n+2}(T_n/I_{n+1}(L_{2,n}))=n-1$ 
and all the other Betti numbers vanish.
\end{theorem}

As Theorem \ref{thm:diamond}, the proof of the above theorem requires a series of preparatory lemmas and propositions. First, we 
define specific vertex and edge labelings of the Hasse diagram of $L_{2,n}$. 
We denote the elements of $L_{2,n}$ by $(0,s)$ and $(1,s)$ for $0\leq s\leq n-1$ and the ordering is given by $(i,j)\leq_{L_{2,n}}  (k,l)$ 
if $i\leq k$ and $j \leq l$. 
Given a cover relation $(i,s)\leq_{L_{2,n}} (j,t)$ in $L_{2,n}$, let $e$ be the corresponding edge in the Hasse diagram of $L_{2,n}$. 
\begin{itemize}
\item[(i)] If $i=j$ and $t=s+1$, then $e$ is labeled with $s^{(i)}$.
\item[(ii)] If $i\neq j$ and $s=t$, then $e$ is labeled with $s^{(c)}$.
\end{itemize}
Edges with label $s^{(c)}$ for some $0\leq s\leq n-1$ are referred to as \emph{connecting edges}. 
(See Figure \ref{chains4} for the labelings for $L_{2,4}$.)
In the sequel, we will identify an edge of the Hasse diagram of $L_{2,n}$ with its label. 

We proceed with the construction of certain subposets of $L_{2,n}$. 
Let $(i_1,\ldots,i_m)$ be a sequence of positive integers with $i_j\geq 2$ for $1\leq j\leq m$ and $\sum_{j=1}^m i_j-m+1=n$. 
Let $L(i_1,\ldots,i_m)$ be the subposet of $L_{2,r}$ obtained by deleting all connecting edges $s^{(c)}$ in the Hasse diagram of 
$L_{2,n}$ for $s\notin \{0,i_1-1,i_1+i_2-2,\ldots,\sum_{j=1}^mi_j-m\}$. For example, if $(i_1,\ldots,i_{n-1})=(2,\ldots, 2)$, then $L(i_1,\ldots,i_{n-1})$ is the 
whole poset $L_{2,n}$. 
Figure \ref{chains2} depicts the Hasse diagram of $L(2,3,2)$ as a subposet of $L_{2,5}$ with the induced vertex and edge labelings. 
\begin{figure}[h]
\begin{center}
\includegraphics[width=6.5cm]{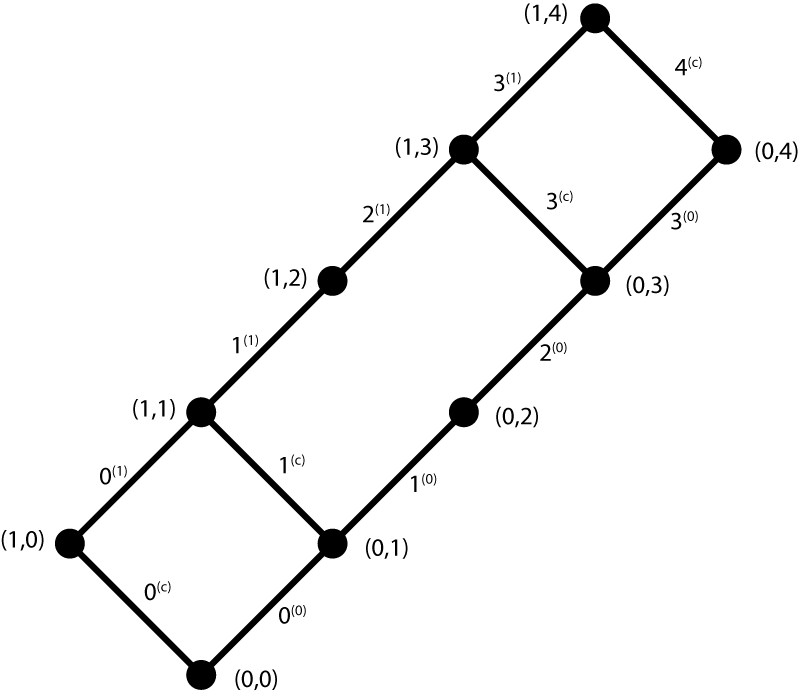}
\end{center}
\caption{The Hasse diagram of $L(2,3,2)$ as a subposet of $L_{2,5}$ with the induced vertex and edge labelings.}
\label{chains2}
\end{figure}
Starting from the posets $L(i_1,\ldots,i_m)$, we define certain collections of subposets of $L_{2,n}$. Let $s \geq 2$ and let $t\geq 1$ be integers. 
A subposet $P$ of $L_{2,n}$ belongs to $\mathcal{P}(s,t;n)$ if there exists a sequence $(i_1,\ldots,i_m)\in \mathbb{N}^m$ with $i_j\geq 2$ 
for $1\leq j\leq m$ such that
\begin{itemize}
\item[(i)] $m+1=s$,
\item[(ii)] $\sum_{j=1}^m i_j-m=t-n-s+1$,
\item[(iii)] $P=L_r\oplus L(i_1,\ldots,i_m)\oplus L_{2n-r-t+s}$ for some $0\leq r\leq 2n-t+s-2$.
\end{itemize}
Here, for a poset $Q$, we set $Q\oplus L_0=L_0\oplus Q=Q$. 
One can visualize a poset $P$ in $\mathcal{P}(s,t;n)$ in the following way: 
Choose a subposet of $L_{2,n}$ that is isomorphic to some $L(i_1,\ldots,i_m)$ such that $L(i_1,\ldots,i_m)$ contains $s$ connecting edges 
(this is condition (i)), and such that each strand of $L(i_1,\ldots,i_m)$ consists of $t-n-s+1$ edges (this is condition (ii)). Finally, 
extend $L(i_1,\ldots,i_m)$ by the unique chains, between $(0,0)$ and the minimum of the chosen $L(i_1,\ldots,i_m)$, and 
between the maximum of the chosen $L(i_1,\ldots,i_m)$ and $(1,n-1)$ (this is condition (iii)). Figure~\ref{chainsSubposet} shows a subposet of $L_{2,10}$ in 
$\mathcal{P}(3,16;10)$ with $(i_1,i_2)=(3,3)$ and $r=3$.
\begin{figure}[h]
\begin{center}
\includegraphics[width=8.5cm]{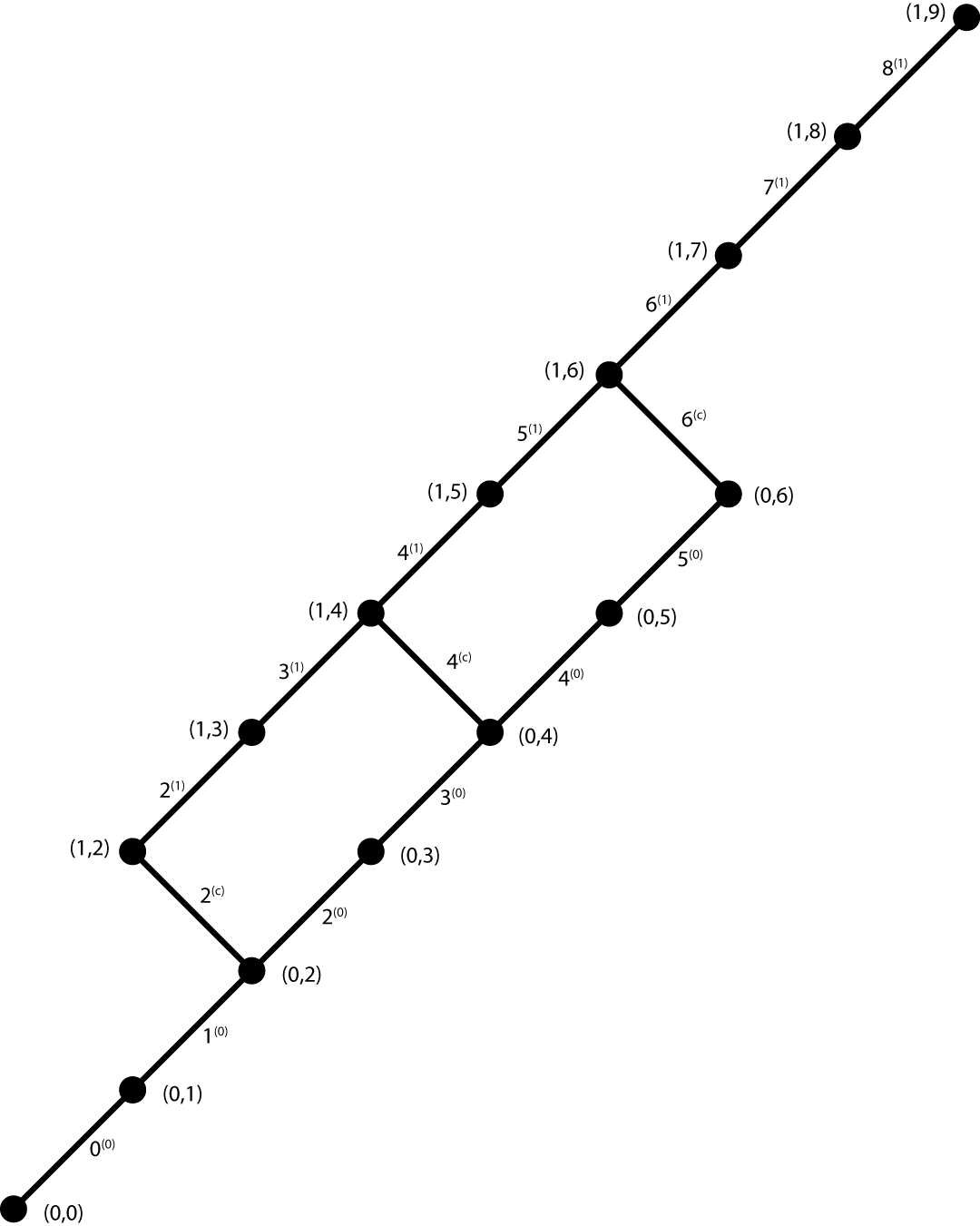}
\end{center}
\caption{A subposet of $L_{2,10}$ in $\mathcal{P}(3,16;10)$ with $(i_1,i_2)=(3,3)$ and $r=3$.}
\label{chainsSubposet}
\end{figure}
Note that $\mathcal{P}(s,t;n)\neq \emptyset$ if and only if $t-n-s-1\leq n-1$ and $s\leq t-n+s+2$, i.\,e., $n+2s-2\leq t\leq 2n+s-2$.

For posets of the form $L(i_1,\ldots,i_m)$, we set 
\begin{equation*}
L^{(0)}=\{0^{(0)},\ldots,(i_1-2)^{(0)}\} \qquad \mbox{ and } \qquad L^{(1)}=\{0^{(1)},\ldots,(i_1-2)^{(1)}\}.
\end{equation*}
Let further
$L(-,i_2,\ldots,i_m)$ be the restriction of the poset $L(i_1,\ldots,i_m)$ to the ground set 
$\{(j,s)~:~j\in\{0,1\},\; i_1-1\leq s\leq \sum_{j=1}^m i_j-m+1\}$, see Figure \ref{chainsRest} for an example.
\begin{figure}[h]
\begin{center}
\includegraphics[width=6.5cm]{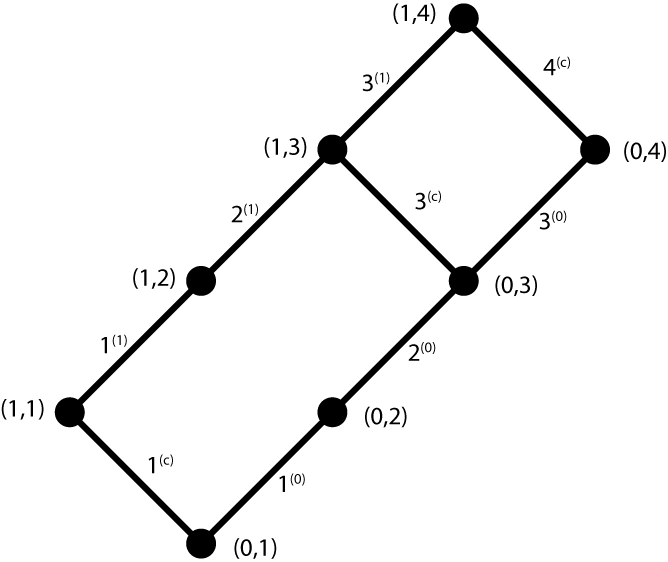}
\end{center}
\caption{The Hasse diagram of $L(-,3,2)$ as subposet of $L(2,3,2)$ with the induced vertex and edge labeling.}
\label{chainsRest}
\end{figure}
Obviously, as posets $L(-,i_2,\ldots,i_m)$ and $L(i_2,\ldots, i_m)$ are isomorphic.
We also define
 $U(i_2,\ldots,i_m)=\{(i_1-1)^{(1)},\ldots,(n-2)^{(1)}\}$  and \\ 
$B(i_2,\ldots,i_m)=\{(i_1-1)^{(0)},\ldots,(n-2)^{(0)}\}\cup\{(i_1-1)^{(c)},(i_1+i_2-2)^{(c)},\ldots, (\sum_{j=1}^m i_j-m)^{(c)}\}$. 
Pictorially, $U(i_2,\ldots,i_m)$ and $B(i_2,\ldots,i_m)$ are the upper respectively the lower (bottom) chain of 
$L(-,i_2,\ldots,i_m)$ including the staves.\\
 In Figure \ref{chainsRest}, the sets $U(3,2)$ and $B(3,2)$ consist of the edges with labels 
in $\{1^{(1)},2^{(1)},3^{(1)}\}$ and $\{1^{(0)},2^{(0)},3^{(0)},1^{(c)},3^{(c)},4^{(c)}\}$, respectively. 
Given a sequence $(i_1,\ldots,i_m)$, we denote by $\Delta_{(i_1,\ldots,i_m)}$ and $\Delta_{(-,i_2,\ldots,i_m)}$ the Stanley-Reisner complex of 
$I_{\LD}(L(i_1,\ldots,i_m))$ and $I_{\LD}(L(-,i_2,\ldots,i_m)$, respectively. 
In the following, if it will not cause confusion, given a poset $P$, we will also use $P$ to denote the set of cover relations of $P$, e.\,g., 
$2^P$ denotes a simplex whose vertices are the cover relations of $P$. 

The proof of our main result will eventually follow from the next proposition and the subsequent three lemmas.

\begin{proposition}\label{prop:homology}
Let $k$ be an arbitrary field. Let $(i_1,\ldots,i_m)$ be a sequence of positive integers with $i_j\geq 2$ for $1\leq j\leq m$ and 
let $n=\sum_{j=1}^m i_j-m+1$. Then
\begin{itemize}
\item[(i)] \begin{align*}
\widetilde{H}_i(\Delta_{(i_1,\ldots,i_m)};k)=
\begin{cases}
k, &\mbox{ if } i=2n-3\\
0, &\mbox{ otherwise}.
\end{cases}
\end{align*}
\item[(ii)] For $m\geq 1$
\begin{align*}
\widetilde{H}_i(\lk_{\Delta_{(i_1,\ldots,i_m)}}(0^{(c)});k)=
\begin{cases}
k, &\mbox{ if } i=2n-4\\
0, &\mbox{ otherwise.}\\
\end{cases}
\end{align*}
\end{itemize}
\end{proposition}

\begin{proof}
We prove (i) and (ii) simultaneously by induction on $n$. By assumption, we have $n\geq 2$. If $n=2$, 
then we must have $m=1$, $i_1=2$ and $L(2)=L_{2,2}$. Hence, $\Delta_{(2)}$ is a $4$-gon and $\lk_{\Delta_{2}}(0^{(c)})$ consists of two isolated 
points. Both, (i) and (ii), follow in this case.

Now let $n\geq 3$. If $m=1$, then $L(i_1)$ consists of two edge-disjoint maximal chains. 
Being those the minimal non-faces of $\Delta_{(i_1)}$, we infer  
\begin{align*}
\Delta_{(i_1)}&=\partial(2^{\{0^{(0)},\ldots,(n-2)^{(0)},(n-1)^{(c)}\}})\ast \partial(2^{\{0^{(c)},0^{(1)},\ldots,(n-2)^{(1)}\}})\mbox{ and } \\
\lk_{\Delta_{(i_1)}}(0^{(c)})&=\partial(2^{\{0^{(1)},\ldots,(n-2)^{(1)}\}})\ast \partial(2^{\{0^{(0)},\ldots,(n-2)^{(0)},(n-1)^{(c)}\}}).
\end{align*}
Theorem \ref{thm:join} implies that $\Delta_{(i_1)}$ and $\lk_{\Delta_{i_1}}(0^{(c)})$ have the homology of a $(2n-3)$- and $(2n-4)$-sphere, respectively. 

Now let $n\geq 3$ and $m\geq 2$. We first show part (ii). 
We can decompose $\lk_{\Delta_{(i_1,\ldots,i_m)}}(0^{(c)})$ in the following way
\begin{align*}
\lk_{\Delta_{(i_1,\ldots,i_m)}}(0^{(c)})=& \left(\partial(2^{L^{(0)}})\ast \partial(2^{L^{(1)}})\ast 2^{L(-,i_2,\ldots,i_m)}\right)\\
\cup & \left(\partial(2^{L^{(0)}})\ast 2^{L^{(1)}}\ast 2^{B(i_2,\ldots,i_m)}\ast\partial(2^{U(i_2,\ldots,i_m)})\right)\\
\cup & \left(2^{L^{(0)}}\ast \partial(2^{L^{(1)}})\ast \Delta_{(-,i_2,\ldots,i_m)}\right)\\
\cup & \left(2^{L^{(0)}}\ast 2^{L^{(1)}}\ast \{i_1^{(c)}\}\ast\lk_{\Delta_{(-,i_2,\ldots,i_m)}}(i_1^{(c)})\right).
\end{align*}
In the above decomposition, the first set accounts for all facets of $\lk_{\Delta_{(i_1,\ldots,i_m)}}(0^{(c)})$, which only contain part of $L^{(0)}$ and $L^{(1)}$, 
the second and third set contain those maximal faces that contain either $L^{(1)}$ or $L^{(0)}$ completely but miss at least one 
element of the other one, and the fourth set assembles those maximal faces that contain both of $L^{(0)}$ and $L^{(1)}$ completely. To simplify notation, we set 
\begin{align*}
M_1^{(1)}&=\partial(2^{L^{(0)}})\ast \partial(2^{L^{(1)}})\ast 2^{L(-,i_2,\ldots,i_m)}\\
M_1^{(2)}&= \partial(2^{L^{(0)}})\ast2^{L^{(1)}}\ast 2^{B(i_2,\ldots,i_m)}\ast\partial(2^{U(i_2,\ldots,i_m)})\\
M_2^{(1)}&=2^{L^{(0)}}\ast \partial(2^{L^{(1)}})\ast \Delta_{(-,i_2,\ldots,i_m)} \\
M_2^{(2)}&=2^{L^{(0)}}\ast 2^{L^{(1)}}\ast \{i_1^{(c)}\}\ast\lk_{\Delta_{(-,i_2,\ldots,i_m)}}(i_1^{(c)})
\end{align*}
and let $M_1=M_1^{(1)}\cup M_1^{(2)}$ and $M_2=M_2^{(1)}\cup M_2^{(2)}$. We will compute the homology of 
$\lk_{\Delta_{(i_1,\ldots,i_m)}}(0^{(c)})$ using the Mayer-Vietoris sequence for $(L_1,L_2)$. We have
\begin{align*}
M_1^{(1)}\cap M_1^{(2)}&=\partial(2^{L^{(0)}})\ast \partial(2^{L^{(1)}})\ast 2^{B(i_2,\ldots,i_m)}\ast \partial(2^{U(i_2,\ldots,i_m)}) \mbox{ and }\\
M_2^{(1)}\cap M_2^{(2)}&=2^{L^{(0)}}\ast\partial(2^{L^{(1)}})\ast\{i_1^{(c)}\}\ast\lk_{\Delta_{(-,i_2,\ldots,i_m)}}(i_1^{(c)}).
\end{align*}
We see that the complexes $M_1^{(1)}$, $M_1^{(2)}$ and $M_1^{(1)}\cap M_1^{(2)}$, such as $M_2^{(1)}$, $M_2^{(2)}$ and $M_2^{(1)}\cap M_2^{(2)}$ 
are all contractible and thus have trivial homology. 
The Mayer-Vietoris sequences \eqref{eq:MV} for $(M_1^{(1)},M_1^{(2)})$ and $(M_2^{(1)},M_2^{(2)})$ imply 
that the same is true for $M_1$ and $M_2$. It remains to compute the homology of $M_1\cap M_2$. In order to determine this intersection, we first compute
\begin{align*}
M_1^{(1)}\cap M_2^{(1)}&=\partial(2^{L^{(0)}})\ast \partial(2^{L^{(1)}})\ast \Delta_{(-,i_2,\ldots,i_m)}\\
M_1^{(1)}\cap M_2^{(2)}&=\partial(2^{L^{(0)}})\ast\partial(2^{L^{(1)}})\ast\{i_1^{(c)}\}\ast \lk_{\Delta_{(-,i_2,\ldots,i_m)}}(i_1^{(c)})\\
M_1^{(2)}\cap M_2^{(1)}&=\partial(2^{L^{(0)}})\ast\partial(2^{L^{(1)}})\ast\{i_1^{(c)}\}\ast \lk_{\Delta_{(-,i_2,\ldots,i_m)}}(i_1^{(c)})\\
M_1^{(2)}\cap M_2^{(2)}&=\partial(2^{L^{(0)}})\ast 2^{L^{(1)}}\ast \{i_1^{(c)}\}\ast \lk_{\Delta_{(-,i_2,\ldots,i_m)}}(i_1^{(c)}).
\end{align*}
We thus obtain
\begin{equation*}
M_1\cap M_2=\partial(2^{L^{(0)}})\ast \partial(2^{L^{(1)}})\ast \Delta_{(-,i_2,\ldots,i_m)}
\cup\partial(2^{L^{(0)}})\ast 2^{L^{(1)}}\ast\{i_1^{(c)}\}\ast \lk_{\Delta_{(-,i_2,\ldots,i_m)}}(i_1^{(c)}).
\end{equation*}
Moreover, the intersection of the two complexes on the right-hand side of the above equation equals
\begin{equation*}
\partial(2^{L^{(0)}})\ast \partial(2^{L^{(1)}})\ast\{i_1^{(c)}\}\ast \lk_{\Delta_{(-,i_2,\ldots,i_m)}}(i_1^{(c)}).
\end{equation*}
This complex as well as $\partial(2^{L^{(0)}})\ast 2^{L^{(1)}}\ast\{i_1^{(c)}\}\ast \lk_{\Delta_{(-,i_2,\ldots,i_m)}}(i_1^{(c)})$ 
are cones and thus have trivial homology. We infer from the Mayer-Vietoris sequence \eqref{eq:MV} for 
$(\partial(2^{L^{(0)}})\ast \partial(2^{L^{(1)}})\ast \Delta_{(-,i_2,\ldots,i_m)},\partial(2^{L^{(0)}})\ast 2^{L^{(1)}}\ast\{i_1^{(c)}\}\ast \lk_{\Delta_{(-,i_2,\ldots,i_m)}}(i_1^{(c)}))$ that 
\begin{equation}\label{eq:inter}
\widetilde{H}_i(M_1\cap M_2;k)\cong \widetilde{H}_{i}(\partial(2^{L^{(0)}})\ast \partial(2^{L^{(1)}})\ast \Delta_{(-,i_2,\ldots,i_m)};k)
\end{equation}
 for all $i$. Theorem \ref{thm:join} implies that
 \begin{align}\label{eq:hom1}
&\empty\widetilde{H}_{i}(\partial(2^{L^{(0)}})\ast \partial(2^{L^{(1)}})\ast \Delta_{(-,i_2,\ldots,i_m)};k)\notag\\
&=\bigoplus_{s+t=i-1}\widetilde{H}_s(\partial(2^{L^{(0)}})\ast \partial(2^{L^{(1)}});k)\otimes_k \widetilde{H}_t(\Delta_{(-,i_2,\ldots,i_m)};k).
\end{align}
Since $L(-,i_2,\ldots,i_m)$ and $L(i_2,\ldots,i_m)$ are isomorphic as posets, $\Delta_{(-,i_2,\ldots,i_m)}$ and $\Delta_{(i_2,\ldots,i_m)}$ are 
simplicially isomorphic and by the induction hypothesis we have
$\widetilde{H}_i(\Delta_{(-,i_2,\ldots,i_m)};k)=k$ if $i=2(\sum_{l=2}^{m} i_l-(m-1)+1)-3=2 \sum_{l=2}^m i_l -2m+1$ and all other homology groups vanish. 
Combining this with the fact that $\partial(2^{L^{(0)}})\ast \partial(2^{L^{(1)}})$ is homeomorphic to a $(2i_1-5)$-sphere, we deduce 
from \eqref{eq:hom1} that 
\begin{equation*}
\widetilde{H}_i(M_1\cap M_2;k)=\widetilde{H}_{i}(\partial(2^{L^{(0)}})\ast \partial(2^{L^{(1)}})\ast \Delta_{(-,i_2,\ldots,i_m)};k)=k
\end{equation*}
 for $i=1+(2i_1-5)+(2 \sum_{l=2}^m i_l -2m+1)=2n-5$ and trivial, otherwise. 
Finally, the Mayer-Vietoris sequence \eqref{eq:MV} for $(M_1,M_2)$ implies 
\begin{align*}
\widetilde{H}_i(\lk_{\Delta_{(i_1,\ldots,i_m)}}(0^{(c)});k)\cong \widetilde{H}_{i-1}(M_1\cap M_2;k)=
\begin{cases}
k, \quad &\mbox{ if } i=2n-4\\
0, \quad &\mbox{ otherwise.}\\
\end{cases}
\end{align*}
This completes the proof of part (ii).\\

We now prove part (i). 
Let $F$ be a face of $\Delta_{(i_1,\ldots,i_m)}$. If $0^{(c)}\in F$, then $F\in \lk_{\Delta_{(i_1,\ldots,i_m)}}(0^{(c)})\ast \{0^{(c)}\}$. 
If $0^{(c)}\notin F$, then also $F\cup L^{(1)}\in\Delta_{(i_1,\ldots,i_m)}$. If $L^{(0)}\subseteq F$, then we must have $F\setminus
(L^{(0)}\cup L^{(1)})\in \Delta_{(-,i_2,\ldots,i_m)}$. 
If instead $L^{(0)}\not\subseteq F$ $F$, then $F\cup L(-,i_2,\ldots,i_m)\in\Delta_{i_1,\ldots,i_m}$. The conducted discussion yields 
the following decomposition of $\Delta_{(i_1,\ldots,i_m)}$
\begin{align*}
\Delta_{(i_1,\ldots,i_m)}&=\lk_{\Delta_{(i_1,\ldots,i_m)}}(0^{(c)})\ast\{0^{(c)}\}\\
&\cup 2^{L(-,i_2,\ldots,i_m)\cup L^{(1)}}\ast \partial(2^{L^{(0)}})\cup 2^{L^{(0)}\cup L^{(1)}}\ast \Delta_{(-,i_2,\ldots,i_m)}.
\end{align*}
For brevity, set
$T_1^{(1)}=\lk_{\Delta_{(i_1,\ldots,i_m)}}(0^{(c)})\ast\{0^{(c)}\}$, $T_1^{(2)}=2^{L(-,i_2,\ldots,i_m)\cup L^{(1)}}\ast \partial(2^{L^{(0)}})$, 
$T_1=T_1^{(1)}\cup T_1^{(2)}$ and $T_2= 2^{L^{(0)}\cup L^{(1)}}\ast \Delta_{(-,i_2,\ldots,i_m)}$.
We want to use the Mayer-Vietoris sequence \eqref{eq:MV} for $(T_1,T_2)$ in order to compute the homology of $\Delta_{(i_1,\ldots,i_m).}$ 
First, note that the intersection of $T_1^{(1)}$ and $T_1^{(2)}$ can be written as
\begin{equation*}
T_1^{(1)}\cap T_1^{(2)}=\partial(2^{L^{(1)}\cup U(i_2,\ldots,i_m)})\ast \partial(2^{L^{(0)}})\ast 2^{B(i_2,\ldots,i_m)}.
\end{equation*}
Since the complexes, $T_1^{(1)}$, $T_1^{(2)}$ and $T_1^{(1)}\cap T_1^{(2)}$ are all cones, it follows from the Mayer-Vietoris sequence \eqref{eq:MV} for 
$(T_1^{(1)},T_1^{(2)})$, that $\widetilde{H}_i(T_1;k)=0$ for all $i$. Clearly, we also have $\widetilde{H}_i(T_2;k)=0$ for all $i$. 
It remains to determine the intersection $T_1\cap T_2$. First observe that
\begin{align*}
T_1^{(1)}\cap T_2&= 2^{L^{(0)}}\ast \partial(2^{L^{(1)}})\ast \Delta_{(-,i_2,\ldots,i_m)}\\
&\cup 2^{L^{(0)}\cup L^{(1)}}\ast \{i_1^{(c)}\}\ast \lk_{\Delta_{(-,i_2,\ldots,i_m)}}(i_1^{(c)}) \mbox{ and }\\
T_1^{(2)}\cap T_2 &=\partial(2^{L^{(0)}})\ast 2^{L^{(1)}}\ast\Delta_{(-,i_2,\ldots,i_m)}.
\end{align*}
We set\\
$U_1= \left(2^{L^{(0)}}\ast\partial(2^{L^{(1)}})\cup\partial(2^{L^{(0)}})\ast 2^{L^{(1)}}\right)\ast\Delta_{(-,i_2,\ldots,i_m)}=\partial(2^{L(0)\cup L(1)})\ast\Delta_{(-,i_2,\ldots,i_m)}$\\
and
$U_2=2^{L^{(0)}\cup L^{(1)}}\ast\{i_1^{(c)}\}\ast\lk_{\Delta_{(-,i_2,\ldots,i_m)}}(i_1^{(c)})$.
Obviously, $T_1\cap T_2=U_1\cup U_2$ and the yet to compute homology of $T_1\cap T_2$ can be deduced from the Mayer-Vietoris sequence \eqref{eq:MV} for 
 $(U_1,U_2)$. 
As already used in the proof of the second claim, the complex $\Delta_{(-,i_2,\ldots,i_m)}$ is simplicially isomorphic to $\Delta_{(i_2,\ldots,i_m)}$. 
The induction hypothesis combined with the fact that $\partial(2^{L(0)\cup L(1)})$ is a $(2i_1-4)$-sphere and Theorem \ref{thm:join} thus imply
{\small{\begin{align*}
\widetilde{H}_i(U_1;k)=
\begin{cases}
k,\qquad &\mbox{ if } i=(2i_1-4)+(2(\sum_{l=2}^m i_l-(m-1)+1)-3)+1=2n-4\\
0, \qquad &\mbox{ otherwise.}\\
\end{cases}
\end{align*}}}
Moreover,
\begin{equation*}
U_1\cap U_2= \partial(2^{L^{(0)}\cup L^{(1)}})\ast\{i_1^{(c)}\}\ast\lk_{\Delta_{(-,i_2,\ldots,i_m)}}(i_1^{(c)}).
\end{equation*}
We hence see that $U_1\cap U_2$ and $U_2$ are cones and thus the Mayer-Vietoris sequence \eqref{eq:MV} for $(U_1,U_2)$ implies that
\begin{align*}
\widetilde{H}_i(T_1\cap T_2;k)=\widetilde{H}_i(U_1\cup U_2;k)=
\begin{cases}
k, \qquad &\mbox{ if } i=2n-4\\
0, \qquad & \mbox{ otherwise.}
\end{cases}
\end{align*}
Finally the Mayer-Vietoris sequence \eqref{eq:MV} for $(T_1,T_2)$ yields
\begin{align*}
\widetilde{H}_i(\Delta_{(i_1,\ldots,i_m)};k)\cong\widetilde{H}_{i-1}(T_1\cap T_2;k)=
\begin{cases}
k, \qquad &\mbox{ if } i=2n-3\\
0, \qquad & \mbox{ otherwise.}
\end{cases}
\end{align*}
This finishes the proof of part (i). 
\end{proof}

In the sequel, we write $\Delta_{\LD,n}$ for the Stanley-Reisner complex of $I_{\LD}(L_{2,n})$. 

\begin{lemma}\label{lemma1}
Let $s,t,n$ be positive integers such that $\mathcal{P}(s,t;n)\neq \emptyset$. If $P\in \mathcal{P}(s,t;n)$, then
\begin{align*}
\widetilde{H}_i((\Delta_{\LD,n})_P;k)=
\begin{cases}
k, &\mbox{ if } i=t-s-1\\
0, &\mbox{ otherwise.}\\
\end{cases}
\end{align*}
\end{lemma}

\begin{proof}
Let $P\in\mathcal{P}(s,t;n)$ and let $(i_1,\ldots,i_m)$ as in the definition of the set $\mathcal{P}(s,t;n)$. 
Let $E(i_1,\ldots,i_m)$ be the edges of $P$ lying in the part that is isomorphic to $C(i_1,\ldots,i_m)$ and 
let $C$ be the remaining edges of $P$. 
The complex $(\Delta_{\LD,n})_P$ can be decomposed in the following way
\begin{equation*}
(\Delta_{\LD,n})_P=\partial(2^{C})\ast 2^{E(i_1,\ldots,i_m)}\cup 2^C \ast \Delta_{(i_1,\ldots,i_m)},
\end{equation*}
where $\Delta_{\LD,E(i_1,\ldots,i_m)}$ denotes the Stanley-Reisner complex of $I_{\LD}(E(i_1,\ldots,i_m))$. 
We set $V_1=\partial(2^{C})\ast 2^{E(i_1,\ldots,i_m)}$ and $V_2= 2^C \ast \Delta_{\LD,E(i_1,\ldots,i_m)}$. 
Since the complexes $V_1$ and $V_2$ are contractible, we infer from the Mayer-Vietoris sequence \eqref{eq:MV} for $(V_1,V_2)$ that
$\widetilde{H}_i((\Delta_{\LD,n})_P;k)=\widetilde{H}_{i-1}(V_1\cap V_2;k)$ for all $i$. Moreover, the last intersection can be computed as
$V_1\cap V_2=\partial(2^C)\ast \Delta_{\LD,E(i_1,\ldots,i_m)}$. 
Since $\Delta_{\LD,E(i_1,\ldots,i_m)}$ is isomorphic to $\Delta_{(i_1,\ldots,i_m)}$, it follows from Proposition \ref{prop:homology} 
that $\widetilde{H}_i(\Delta_{\LD,E(i_1,\ldots,i_m)};k))=k$ if $i=2(t-n-s+2)-3=2(t-n-s)+1$ and 
$\widetilde{H}_i(\Delta_{\LD,E(i_1,\ldots,i_m)};k))=0$ otherwise. Since $\partial(2^C)$ is the boundary of a $(2n-t+s-3)$-simplex, 
Theorem \ref{thm:join} implies that
\begin{align*}
\widetilde{H}_{i}((\Delta_{\LD,n})_P;k)=\widetilde{H}_i(V_1\cap V_2;k)=
\begin{cases}
&k, \quad\mbox{ if } i=t-s-1\\
&0, \quad \mbox{ otherwise.} 
\end{cases}
\end{align*}
\end{proof}

\begin{lemma}\label{lemma2}
If $P$ is a subposet of $L_{2,n}$ not contained in any $\mathcal{P}(s,t;n)$ for $s\geq 2$ and $n+2s-2\leq t\leq 2n+s-2$, 
then $\widetilde{H}_i((\Delta_{\LD,n})_P;k)=0$ for all $i$.
\end{lemma}

\begin{proof}
Assume that $P$ is a subposet of $L_{2,n}$ such that $\widetilde{H}_i((\Delta_{\LD,n})_P;k)\neq 0$ for some $i$. Let 
\begin{equation*}
t_1=\min\{r~:~r^{(c)} \mbox{ is an edge of the Hasse diagram of }P\}
\end{equation*}
and let 
\begin{equation*}
t_2=\max\{r~:~r^{(c)} \mbox{ is an edge of the Hasse diagram of }P\}.
\end{equation*}
Note that $t_1$ and $t_2$ exist since $P$ must contain at least one maximal chain of $L_{2,n}$ (otherwise $(\Delta_{\LD,n})_P$ would 
have trivial homology). By the same reason, we know that each edge of $P$ has to lie in at least one maximal chain of $P$, that is also a 
maximal chain of $L_{2,n}$. It hence follows from the definition of $t_1$ and $t_2$ that $P$ contains the edges with labels $0^{(0)},\ldots,(t_2-1)^{(0)}$ 
but it does neither contain the edges with labels $0^{(1)},\ldots,(t_1-1)^{(1)}$ nor those with labels $0^{(c)},\ldots (t_1-1)^{(c)}$. 
Similarly, $P$ contains the edges with labels $t_1^{(1)},\ldots,(n-1)^{(1)}$ but it does neither contain the edges with labels 
$t_2^{(0)},\ldots,(n-1)^{(0)}$ nor those with labels $(t_2+1)^{(c)},\ldots (n-1)^{(c)}$. This already implies that $P$ belongs to 
some $\mathcal{P}(s,t;n)$ and the claim follows.
\end{proof}

\begin{lemma}\label{lemma3}
 $|\mathcal{P}(s,t;n)|=(2n-t+s-1)\cdot \binom{t-n-s}{s-2}$ for $s\geq 2$ and $n+2s-2\leq t\leq 2n+s-2$.
 \end{lemma}

\begin{proof}
In order to construct a poset $P\in \mathcal{P}(s,t;n)$, we can first choose a poset $C(i_1,\ldots,i_m)$. Since such a poset has to 
contain $t-n-s+1$ edges on each strand and since there have to be $s$ connecting edges, the only freedom we have, is to choose the positions 
of those connecting edges. However, the upper and lower edge are fixed and we can only distribute the $(s-2)$ remaining connecting edges over 
$t-n-s$ positions. Hence, there are $\binom{t-n-s}{s-2}$ ways to choose $C(i_1,\ldots,i_m)$. The next step is to choose the position of the poset 
$C(i_1,\ldots,i_m)$ inside the poset $L_{2,n}$. Since each strand of $C(i_1,\ldots,i_m)$ and $L_{2,n}$ consists of $t-n+s-1$ and 
$n-1$ edges, respectively, there are $(n-1)-(t-n-s+1)+1=2n-t+s-1$ possible choices. 
\end{proof}

Finally, we can give the proof of Theorem \ref{thm:chains}.

{\sf Proof of Theorem \ref{thm:chains}:}
Using the Hochster formula, it follows from Lemma \ref{lemma1} and Lemma \ref{lemma2} that $\beta_{i,j}(S_n/I_{\LD}(L_{2,n}))=|\mathcal{P}(i,j;n)|$ and 
by Lemma \ref{lemma3} we can conclude that $\beta_{i,j}(S_n/I_{\LD}(L_{2,n}))=(2n-j+i-1)\cdot \binom{j-n-i}{i-2}$ for $i\geq 2$ and 
$n+2i-2\leq j\leq 2n+i-2$. Note that $\beta_i(S_n/I_{\LD}(L_{2,n}))$ counts the number of posets in $\bigcup_{j=n+2i-2}^{2n+i-2}\mathcal{P}(i,j;n)$.
Those are all posets lying in some $\mathcal{P}(s,t;n)$ having exactly $i$ connecting edges. Since each such poset is determined by the position of those 
edges and since we have $\binom{n}{i}$ possiblities to position $i$ edges among $n$ positions, we obtain $\beta_i(S_n/I_{\LD}(L_{2,n}))=\binom{n}{i}$. 

It remains to show the statements for the path ideals $I_{n+1}(L_{2,n})$. Consider the minimal set of generators of $I_{n+1}(L_{2,n})$. For 
$1\leq i\leq n$, let $u_i$ be the generator corresponding to the chain $0^{(0)},\ldots,(i-2)^{(0)},(i-1)^{(c)},(i-1)^{(1)},\ldots,(n-2)^{(1)}$. 
 It is easy to check that for $2\leq i\leq n$, the colon ideal 
$(u_1,\ldots,u_{i-1}):(u_i)$ is generated by the variable corresponding to the element $(1,i-1)\in L_{2,n}$. Hence, $I_{n+1}(L_{2,n})$ has linear quotients and since
it is generated in degree $n+1$, we can infer from \cite{HerzogTakayama} (see also Corollary 1.3.5 in \cite{OlteanuPhD}), that $\beta_{1,n+1}(T_n/I_{n+1}(L_{2,n}))=n$, 
$\beta_{2,n+2}(T_n/I_{n+1}(L_{2,n}))=n-1$ and $\pd(T_n/I_{n+1}(L_{2,n}))=2$ and all other Betti numbers --~except for $\beta_{0,0}(T_n/I_{n+1}(L_{2,n}))=1$~-- 
vanish.
\qed

Using Theorem \ref{thm:chains}, we can compute the projective dimension, the regularity and the depth of $S_n/I_{\LD}(L_{2,n})$ and $T_n/I_{n+1}(L_{2,n})$.
\begin{corollary}\label{cor:chains}
Let $n\geq 2$ be a positive integer. Then:
\begin{itemize}
\item[(i)] $\pd(S_n/I_{\LD}(L_{2,n}))=n$ and $\pd(T_n/I_{n+1}(L_{2,n}))=2$,
\item[(ii)] $\reg(S_n/I_{\LD}(L_{2,n}))=2n-2$ and $\reg(T_n/I_{n+1}(L_{2,n}))=n$,
\item[(iii)] $\depth(S_n/I_{\LD}(L_{2,n}))=2n-2$ and $\depth(T_n/I_{n+1}(L_{2,n}))=2n-2$.
\end{itemize}
\end{corollary}

\begin{proof}
Parts (i) and (iii) follow by exactly the same arguments as in the proof of Theorem \ref{thm:diamond}. 
The statement from part (ii) is a direct consequence of Theorem \ref{thm:chains}. In particular, this theorem shows that for $S_n/I_{\LD}(L_{2,n})$ the regularity is 
already attained in the second step of the resolution. 
\end{proof}

\begin{lemma}\label{lem:chains}
Let $n\geq 2$ be a positive integer. Then: 
\begin{itemize}
\item[(i)] $\height(I_{\LD}(L_{2,n}))=2$ and $\height(I_{n+1}(L_{2,n}))=$, 
\item[(ii)] $\dim(S_n/I_{\LD}(L_{2,n}))=3n-4$ and $\dim(T_n/I_{n+1}(L_{2,n}))=2n-1$.
\end{itemize}
\end{lemma}

\begin{proof}
Note that the edges with labels $0^{(0)}$ and $0^{(1)}$ form a minimal cut-set of the Hasse diagram of $L_{2,n}$. 
Proposition \ref{prop:primary} implies that $\height(I_{\LD}(L_{2,n}))=2$ and since there are $3n-2$ edges in the Hasse diagram of $L_{2,n}$ 
it holds that $\dim(S_n/I_{\LD}(L_{2,n}))=3n-4$.\\
For $T_n/I_{n+1}(L_{2,n})$, note that both, the minimal and the maximal element, of $L_{2,n}$ are a minimal vertex cover for $\Delta_{n+1}(L_{2,n})$. 
The claim follows from the discussion preceding Proposition \ref{prop:primary}.
\end{proof}

  \bibliography{biblio}
  \bibliographystyle{plain}

\end{document}